\documentclass[12pt]{amsart}
\usepackage{amsmath,amsthm,amssymb,amsfonts,amscd}
\usepackage{mathrsfs}
\usepackage{bbding}
\usepackage{graphicx,latexsym}
\usepackage[backref=page]{hyperref} 
\usepackage{geometry}
\usepackage{color}

\numberwithin{equation}{section}

\theoremstyle{plain}
\newtheorem{theorem}[equation]{Theorem}
\newtheorem{lemma}[equation]{Lemma}

\newtheorem{claim}[equation]{Claim}

\theoremstyle{definition}
\newtheorem{definition}[equation]{Definition}

\theoremstyle{remark}
\newtheorem{remark}[equation]{Remark}

\renewcommand{\Re}{\operatorname{Re}}
\renewcommand{\Im}{\operatorname{Im}}

\newcommand{\sym}{\operatorname{sym}}

\renewcommand{\mod}{\operatorname{mod}\ }

\newcommand{\cA}{\mathcal{A}}
\newcommand{\cB}{\mathcal{B}}

\newcommand{\cD}{\mathcal{D}}

\newcommand{\cF}{\mathcal{F}}

\newcommand{\cM}{\mathcal{M}}
\newcommand{\cN}{\mathcal{N}}

\newcommand{\sV}{\mathscr{V}}

\newcommand{\ve}{\varepsilon}

\makeatletter
\def\@tocline#1#2#3#4#5#6#7{\relax
  \ifnum #1>\c@tocdepth 
  \else
    \par \addpenalty\@secpenalty\addvspace{#2}%
    \begingroup \hyphenpenalty\@M
    \@ifempty{#4}{%
      \@tempdima\csname r@tocindent\number#1\endcsname\relax
    }{%
      \@tempdima#4\relax
    }%
    \parindent\z@ \leftskip#3\relax \advance\leftskip\@tempdima\relax
    \rightskip\@pnumwidth plus4em \parfillskip-\@pnumwidth
    #5\leavevmode\hskip-\@tempdima
      \ifcase #1
       \or\or \hskip 1em \or \hskip 2em \else \hskip 3em \fi%
      #6\nobreak\relax
    \hfill\hbox to\@pnumwidth{\@tocpagenum{#7}}\par
    \nobreak
    \endgroup
  \fi}
\makeatother

\begin{document}

\title[Super-positivity]
        {Super-positivity of a family of L-functions \\ in the level aspect}
\author{Dorian Goldfeld and Bingrong Huang}
\address{Dorian Goldfeld: Department of Mathematics \\ Columbia University \\ New York \\NY 10027 \\USA}
\email{goldfeld@columbia.edu}
\address{Bingrong Huang: School of Mathematics \\ Shandong University \\ Jinan \\Shandong 250100 \\China \newline
\indent Current address: School of Mathematical Sciences \\ Tel Aviv University \\ Tel Aviv 69978 \\ Israel}
\email{bingronghuangsdu@gmail.com}
\thanks{Dorian Goldfeld is partially supported by NSA Grant H98230-16-1-0009,
Bingrong Huang is\break\phantom{xx} supported  by the European Research Council, under the European Union's Seventh Framework\break\phantom{xx} Programme (FP7/2007-2013)/ERC grant agreement n~$^{\text{o}}$~320755}

\dedicatory{For Don Zagier on his 65$^{th}$ birthday.}

\begin{abstract}
 An automorphic self dual L-function has the super-positivity property if all derivatives of the completed L-function at the central point $s=1/2$ are non-negative and all derivatives at a real point $s > 1/2$ are positive. In this paper we prove that at least 12\% of
  L-functions associated to
  Hecke basis cusp forms of weight $2$ and large prime level $q$  have the super-positivity property. It is also shown that at least 49\% of such L-functions have no real zeros on $ \Re(s) > 0$ except possibly at $s = 1/2.$
\end{abstract}

\maketitle

\section{Introduction} \label{sec:intr}

The notion of super-positivity of self dual L-functions was introduced in \cite{yun2015shtukas} with breakthrough applications to the function field analog of the Gross--Zagier formula for higher derivatives of L-functions.
We say a self dual L-function has the \emph{super-positivity} property if all derivatives of its completed L-function (including Gamma factors and power of the conductor) at the central value $s = 1/2$ are non-negative,
all derivatives at a real point $s > 1/2$ are positive,
and if the $k_0$-th derivative is positive then all $(k_0+2j)$-th derivatives are positive, for all $j\in\mathbb{N}$. See \cite{goldfeld2016super} for more details.

 In this paper, we continue our investigation of super-positivity of L-functions associated to classical modular forms
  as begun in \cite{goldfeld2016super}.
Let $q$ be a fixed large prime.
Let $S_2(q)$ be the space of holomorphic weight $2$ cusp forms
of level $q$, and $H_2(q)$ a basis of $S_2(q)$ that are eigenfunctions
of all the Hecke operators and have the first Fourier coefficient
$a_f(1)=1$. Our main theorem in this paper is as follows.

\begin{theorem} \label{thm:SuperPositivity}
  There are infinitely many modular forms $f$ of weight $2$ and prime level
  such that $L(s,f)$ has the super-positivity property.
  In fact, the proportion of $f\in H_2(q)$ which have the super-positivity
  property is $\geq 12\%$, when $q$ is a sufficiently large prime.
\end{theorem}

In the course of proving theorem \ref{thm:SuperPositivity}
we also obtained the following result about real zeros of L-functions as a
by-product.

\pagebreak

\begin{theorem} \label{thm:RealZero}
  There are infinitely many modular forms $f$ of weight $2$ and prime level
  such that $L(s,f)$ has no nontrivial real zeros except possibly at $s=1/2$.
  In fact, the proportion of $f\in H_2(q)$ with this property is $\geq 49\%$,
  when $q$ is a sufficiently large prime.
\end{theorem}

A key ingredient of the proofs of the above theorems is given in a paper of Stark and Zagier \cite{StarkZagier1980} which essentially establishes that super-positivity follows if $L(s,f)$ has no zeros in the triangle $\big\{s=\sigma+it \; \big |\; 1/2<\sigma<1,\; t\in\mathbb R,\; |t| \leq \sigma- 1/2\big\}$.
In this regard, see \cite{goldfeld2016super}.
To prove such a zero-free region for a large number of those L-functions, we first use Selberg's lemma \cite[Lemma 14]{selberg1946contributions2} to convert the count of the number of zeros in this triangle
to estimates of the mollified second moments.
We will need an asymptotic formula for the mollified second moment of L-functions near the central point $1/2$.
In \S \ref{sec:msm}, we prove this (with the harmonic weights) by an asymptotic formula for the twisted second moment (see theorem \ref{thm:TwistedSecondMoment} below) and Perron's formula.
Then we deal with L-functions away from the critical line with the harmonic weights by the convexity principle (see \S \ref{sec:mollification_away}).
In order to get a better bound, we also need to improve an estimate of the first moment, see lemmas \ref{lemma:FirstMomentL_h} and \ref{lemma:FirstMoment}.
The harmonic weights appear naturally, since we will use the Petersson trace formula. We will remove those harmonic weights in \S \ref{sec:removing_weights}.
The proofs of both theorems \ref{thm:SuperPositivity} and \ref{thm:RealZero} depend on numerics, which we didn't try to optimize.
Note that some of the techniques in this paper are motivated
  by Kowalski--Michel \cite{kowalski1999analytic,kowalski2000explicit} and
  Conrey--Soundararajan \cite{conrey2002real,conrey3000real}.

\begin{remark}
  It is possible to consider even and odd forms separately,
  and then get a more precise result even for fixed even weight and square-free level. One can also prove
  a positive proportion result in the weight or spectral aspect
  by combining the ideas in this paper with   \cite{goldfeld2016super}.
\end{remark}

\medskip
\noindent
\textbf{Acknowledgements.} The second author would like to thank Professors Jianya Liu and Ze\'ev Rudnick for their valuable advice and constant encouragement.


\section{Preliminaries}

%
%
%

\subsection{The approximate functional equation}
For $f\in H_2(q)$, let
$$
  f(z) = \sum\limits_{n=1}^\infty n^{1/2} \lambda_f(n) e(nz), \quad \textrm{(for $z\in \mathbb H$)}
$$
be its Fourier expansion.
We can associate it with an L-function
$$L(s, f) := \sum_{n=1}^\infty \frac{\lambda_f(n)}{n^s},\quad \textrm{(for $\Re(s) > 1$)}.$$
The functional equation is written in terms of the completed L-function
\[
  \Lambda(s,f) = \left(\frac{\sqrt{q}}{2\pi}\right)^s \Gamma\left(s+\frac{1}{2}\right)L(s,f),
\]
namely
\begin{equation}\label{eqn:FE}
  \Lambda(s,f) = \ve_f \Lambda(1-s,f),
\end{equation}
and the sign $\ve_f$ of the functional equation is
\begin{equation}\label{eqn:ve_f}
  \ve_f = q^{1/2}\lambda_f(q) \in \{\pm1\}.
\end{equation}

Let $G$ be a polynomial of degree $2N$ (large enough) satisfying
\begin{equation}\label{eqn:G}
  G(-s)=G(s), \qquad    G(0)=1,
  \quad \mathrm{and}\quad
  G(-N)=\ldots=G(-1)=0.
\end{equation}
Notice that we have $G'(0)=G'''(0)=0$.
That is
\[
  G(s)=a_0\prod_{k=1}^{N} (s^2-k^2), \quad a_0=(-1)^N (N!)^{-2}.
\]
Define
\begin{equation}\label{eqn:H_t(s)}
  H_t(s) := G(s+it)G(s-it)\Gamma(s+1+it)\Gamma(s+1-it).
\end{equation}
By Stirling's formula, (in a vertical strip) we have
\begin{equation}\label{eqn:H<<}
  \frac{H_t(s)}{H_t(\delta)} \ll (1+|t|+|\Im(s)|)^{B} \cdot \min\left\{ 1, e^{-\pi (|\Im(s)|-|t|)}\right\},
\end{equation}
for some constant $B>0$ depending on $N$ and $\Re(s)$.
Note that for $|\delta|\ll \frac{\log\log q}{\log q}$, we have
\[
  \frac{H_t(-\delta)}{H_t(\delta)} = \frac{\Gamma(-\delta+1+it)\Gamma(-\delta+1-it)}{\Gamma(\delta+1+it)\Gamma(\delta+1-it)}
  = 1 + \mathcal{O}\big(|\delta|\big).
\]
By considering the following integral
\[
  I_{\delta} = \frac{1}{2\pi i} \int\limits_{3-i\infty}^{3+i\infty} \Lambda(1/2+\delta+it+s,f)\Lambda(1/2+\delta-it+s,f) G(s+it) G(s-it) \frac{ds}{s-\delta},
\]
we can obtain the approximate functional equation of $|L(1/2+\delta+it,f)|^2$.

\begin{lemma}\label{lemma:AFE}
  Let $f\in H_2(q)$.
  Let $-\frac{B}{\log q}\leq \delta\leq \vartheta$, $\delta\neq0$, and $t\in \mathbb{R}$.
  We have
  \[
    \big|L(1/2+\delta+it,f)\big|^2 = \left(\frac{q}{4\pi^2}\right)^{-\delta}
       \sum_{n=1}^{\infty}\frac{\lambda_f(n)\eta_{it}(n)}{n^{1/2}}
       V_{\delta,t}\left(\frac{4\pi^2n}{q}\right).
  \]
 Here $\eta_\nu(n):=\sum\limits_{ad=n}\left(\frac{a}{d}\right)^\nu$ is the generalized divisor function, and for any $y>0,$
  \[
    V_{\delta,t}(y) := \frac{1}{2\pi i} \int\limits_{3-i\infty}^{3+i\infty}
         \frac{H_{t}(s)}{H_{t}(\delta)} \zeta_q(1+2s)y^{-s} \frac{2s}{(s-\delta)(s+\delta)} \; ds
  \]
  is real valued, and satisfies the following:
  \[
    \begin{split}
      V_{\delta,t}(y) & = \zeta_q(1+2\delta)y^{-\delta} + \frac{H_{t}(-\delta)}{H_{t}(\delta)} \zeta_q(1-2\delta)y^{\delta}
      + \mathcal{O}_{N}\left(y^N(1+|t|)^{B}\right),
    \end{split}
  \]
  for $0< y\leq 1$; and
  \[
    V_{\delta,t}(y) \ll_{A,N} y^{-A}(1+|t|)^{B} \quad \textrm{for all}\quad A\geq1,
  \]
  for $y\geq 1$, and  $B$ depending on $A$ and $N$.
  In particular, for all $y\geq 1/q$ we have
  \[
    V_{\delta,t}(y) \ll q^{\delta} (\log q) (1+|t|)^B.
  \]
  Here $\zeta_q$ is the Riemann zeta function with the Euler factor at $q$ removed.
\end{lemma}

\begin{proof}
  See Kowalski--Michel \cite[pp. 526--527]{kowalski1999analytic}.
\end{proof}

\subsection{The Petersson trace formula}

Let $\mathbb{H}$ denote the upper half plane.
The Fourier coefficients of $f\in H_2(q)$ satisfy the relation
\begin{equation}\label{eqn:HR}
  \lambda_f(m)\lambda_f(n)= \sum_{d|(m,n)} \varepsilon_q(d) \lambda_f\left(\frac{mn}{d^2}\right),
\end{equation}
where $\varepsilon_q(\cdot)$ is the principle Dirichlet character modulo $q$.
In particular, we have $\lambda_f(q)\lambda_f(\ell)=\lambda_f(q\ell)$ for any $\ell\in\mathbb{Z}_{>0}$.
The Petersson trace formula is given by the following basic orthogonality relation on $H_2(q)$.

\begin{lemma}\label{lemma:PTF}
  Let $m,n\geq1$. Then
  \[  \sum_{f\in H_2(q)} \omega_f\cdot  \lambda_f(m)\lambda_f(n)
 \;   = \; \delta_{m,n} - 2\pi \sum_{c\equiv0(\mod q)}
                 \frac{S(m,n;c)}{c}J_{1}\left(\frac{4\pi\sqrt{mn}}{c}\right).
  \]
  where
  $$
    \omega_f := \frac{1}{4\pi\langle f,f\rangle}
    = \frac{\zeta(2)}{|H_2(q)|} \cdot \frac{1}{ L(1,\operatorname{sym}^2 f)  }
    \cdot\left(1+ \mathcal{O}\Big(\frac{(\log q)^4}{q}\Big)\right)
  $$
  is termed the harmonic weight of $f$,
  where $\langle\cdot,\cdot\rangle$ denotes the Petersson inner product on $S_2(q)$,
  $$
    \langle f,g \rangle = \int\limits_{\Gamma_0(q)\backslash\mathbb{H}} f(z)\,\overline{g(z)} \;\frac{dxdy}{y^2}.
  $$
\end{lemma}
\begin{proof}
  See e.g. 
  Kowalski--Michel \cite[pp. 507--509]{kowalski1999analytic}.
\end{proof}
By Weil's bound for Kloosterman sums and the bound $J_1(x)\ll x$ for $J$-Bessel function,
we have (see e.g. \cite[eq. (12)]{kowalski1999analytic})
\begin{equation}\label{eq:harmonicsum}
  \sum_{f\in H_2(q)} \omega_f = 1  \;+ \;\mathcal O\left(q^{-3/2} \right).
\end{equation}
We can also consider the number of odd cusp forms.
Note that
\begin{equation}\label{eqn:sum^odd}
  \underset{f\ \textrm{is odd}}{\sum_{f\in H_2(q)}} \alpha_f =  \sum_{f\in H_2(q)} \frac{1-\ve_f}{2} \cdot \alpha_f,
\end{equation}
for any finite set of complex numbers $(\alpha_f)$.
By Kowalski--Michel \cite[Lemma 1]{kowalski2000lower}, we have
\begin{equation}\label{eqn:harmonicsum-odd}
  \underset{f\ \textrm{is odd}}{\sum_{f\in H_2(q)}} \omega_f = \frac{1}{2} + \mathcal{O}\left(\frac{1}{q}\right).
\end{equation}
Also, it's well known that
\begin{equation}\label{eqn:sum-odd}
  \underset{f\ \textrm{is odd}}{\sum_{f\in H_2(q)}} 1
  \sim \frac{1}{2} \sum_{f\in H_2(q)} 1.
\end{equation}


\section{The twisted second moment near the critical point}\label{sec:tsm}

Recall that $H_2(q)$ denotes a basis for the space of holomorphic Hecke cusp forms of weight $k=2$ for $\Gamma_0(q)\subset SL(2,\mathbb Z).$
Assume that for  all $f \in H_2(q)$  there is some uniquely defined $\alpha_f\in\mathbb C.$ Consider the set $$\{\alpha_f\} := \{\alpha_f\}_{f\in H_2(q)}.$$
The basic objects of study for the rest of this paper are given in the following definition.

\begin{definition} \label{defSums}
  Let
  $$
    \mathcal R_{1/2}
      := \big\{\beta+i\gamma \; \big | \; \beta \in (1/2,1) \; \text{and} \; |\gamma|\leq\beta - 1/2\big\}.
  $$
  For $q > 1$, a large prime number, we define the following sums
  \begin{equation*}
    \begin{split}
      &\mathcal{A}\big(\{\alpha_f\};  q\big) \; :=  \sum_{f\in H_2(q)} \alpha_f,
      \qquad
      \mathcal{A}(q) \; := \sum_{f\in H_2(q)} 1, \\
      &\mathcal{M}(q) \; :=  \underset{L(s,f)\, \ne \,0 \; \text{for} \; s \,\in \,\mathcal R_{1/2}} {\sum_{f\in H_2(q)}} \hskip-20pt 1,
      \qquad\quad\
      \mathcal{N}(q) \; := \underset{ \text{one zero in}  \,\mathcal R_{1/2}} {\underset{L(s,f)\; \text{has at least} } {\sum_{f\in H_2(q)}}}  \hskip -10pt 1.
    \end{split}
  \end{equation*}
\end{definition}

It follows from Stark--Zagier  \cite{StarkZagier1980} that it is enough to prove that at least 12\% of L-functions associated to $f \in H_2(q)$ have no zeros in the triangle $\mathcal R_{1/2}.$ Furthermore,
it is clear that
$$
  \mathcal{M}(q)+\mathcal{N}(q)=\mathcal{A}(q).
$$
The key strategy for proving  theorem \ref{thm:SuperPositivity} is
to try to show that $\mathcal{M}(q)$ is large compared to $\mathcal{A}(q)$.
To achieve this goal we will use the mollification method
which leads us to first consider the following twisted second moment of $L(s,f)$ at the special value $s = 1/2+\delta+it$.

\begin{theorem}\label{thm:TwistedSecondMoment}
  Let $-\frac{c}{\log q}\leq \delta \leq \vartheta$  with $c>0, \,0 < \vartheta <  \frac{1}{100}$ both fixed,   and
  $t\in\mathbb{R}$.
  Let $\ell\leq q^{1-4\vartheta}$ be a positive integer.
  We have the following asymptotic formula
 \[
    \begin{split}
       & \mathcal{A} \Big(\left\{\omega_f \lambda_f(\ell)|L(1/2+\delta+it,f)|^2\right\}; \, q\Big) \\
       & \hskip 10pt
       = \zeta_q(1+2\delta) \frac{\eta_{it}(\ell)}{\ell^{1/2+\delta}}
             + \zeta_q(1-2\delta) \frac{\eta_{it}(\ell)}{\ell^{1/2-\delta}}
                \frac{H_t(-\delta)}{H_t(\delta)} \left(\frac{q}{4\pi^2}\right)^{-2\delta}
              +  \mathcal{O}_{\varepsilon}\Big(\big(1+|t|\big)^B q^{-1/2+\ve}\Big).
    \end{split}
 \]
\end{theorem}

\begin{remark}
  One may use a similar method to obtain the asymptotic formula
  for the average over even or odd forms respectively.
  See Kowalski--Michel \cite{kowalski2000lower} for example.
\end{remark}

\begin{proof}

We now begin the proof of theorem \ref{thm:TwistedSecondMoment}.
From the approximate functional equation (lemma \ref{lemma:AFE}), we have
\begin{align*}
       &   \mathcal{A}\Big(\left\{\omega_f \lambda_f(\ell)|L(1/2+\delta+it,f)|^2\right\}; q\Big) =
            \sum_{f\in H_2(q)} \omega_f\cdot \lambda_f(\ell)|L(1/2+\delta+it,f)|^2 \\
       &\hskip 90pt =  \left(\frac{q}{4\pi^2}\right)^{-\delta} \sum_{n=1}^{\infty}\frac{\eta_{it}(n)}{n^{1/2}} V_{\delta,t}\left(\frac{4\pi^2 n}{q}\right)
         \underset{f\in H_2(q)}{{\sum}}\omega_f\cdot \lambda_f(\ell)\lambda_f(n).
\end{align*}

Applying the Petersson trace formula (lemma \ref{lemma:PTF}) we obtain that
\begin{equation}\label{eqn:D+F}
  \mathcal{A}\Big(\left\{\omega_f \lambda_f(\ell)|L(1/2+\delta+it,f)|^2\right\}; \; q\Big) = \cD + \cF,
\end{equation}
where we have the diagonal term (noting that $\ell<q$)
\begin{equation*}
  \cD := \left(\frac{q}{4\pi^2}\right)^{-\delta} \frac{\eta_{it}(\ell)}{\ell^{1/2}} V_{\delta,t}\left(\frac{4\pi^2 \ell}{q}\right),
\end{equation*}
and the off-diagonal term
\begin{equation*}
  \begin{split}
    \cF & :=  -2\pi \left(\frac{q}{4\pi^2}\right)^{-\delta} \sum_{n=1}^{\infty}\frac{\eta_{it}(n)}{n^{1/2}} V_{\delta,t}\left(\frac{4\pi^2 n}{q}\right)
         \sum_{q|c} \frac{S(\ell,n;c)}{c} J_1\left(\frac{4\pi\sqrt{\ell n}}{c}\right).
  \end{split}
\end{equation*}

\subsection{The diagonal term}

Recall that $\vartheta$ is a fixed positive real number less than $1/100$.
We first handle the case $-\frac{c}{\log q}\leq \delta \leq \vartheta$
and $\delta\neq0$. Note that for the remaining case $\delta=0$,
we can just view it as the limitation of $\delta\rightarrow0$.
Introducing the integral defining $V_{\delta,t}(y)$ in lemma \ref{lemma:AFE}, we have
\[
  \begin{split}
    \cD & = \left(\frac{q}{4\pi^2}\right)^{-\delta} \frac{\eta_{it}(\ell)}{\ell^{1/2}} \frac{1}{2\pi i} \int\limits_{3-i\infty}^{3+i\infty}
        \frac{H_{t}(s)}{H_{t}(\delta)}\;  \zeta_q(1+2s) \left(\frac{4\pi^2 \ell}{q}\right)^{-s}
        \frac{2s}{(s-\delta)(s+\delta)} \; ds.
  \end{split}
\]
For $-1/2\leq \Re(s)\leq 3$, 
we have
\[
  \frac{| \Gamma(s+1+it)\Gamma(s+1-it) |}{\Gamma(\delta+1+it)\Gamma(\delta+1-it)}
  \leq \frac{\big|\Gamma\big(\Re(s)+1\big)^2 \big|}{\Gamma(\delta+1)^2} \ll 1.
\]
Hence by shifting the contour to the line $\Re(s)=-1/2+\ve$, we obtain
\begin{equation}\label{eqn:cD=}
  \begin{split}
    \cD & = \zeta_q(1+2\delta) \frac{\eta_{it}(\ell)}{\ell^{1/2+\delta}}
           \;  + \; \zeta_q(1-2\delta) \frac{\eta_{it}(\ell)}{\ell^{1/2-\delta}}
           \frac{H_t(-\delta)}{H_t(\delta)} \left(\frac{q}{4\pi^2}\right)^{-2\delta} \\
        & \hskip 10pt  + \mathcal{O} \Bigg(\frac{\tau(\ell ) q^{-\delta}}{\ell^{1/2}}  \Bigg | \,\int\limits_{-1/2+\ve-i\infty}^{-1/2+\ve+i\infty}
                \frac{H_{t}(s)}{H_{t}(\delta)}  \zeta_q(1+2s) \left(\frac{4\pi^2 \ell}{q}\right)^{-s} \frac{2s}{(s-\delta)(s+\delta)} \; ds \,\Bigg |\;\Bigg) \\
        & = \zeta_q(1+2\delta) \frac{\eta_{it}(\ell)}{\ell^{1/2+\delta}}
           \;  + \; \zeta_q(1-2\delta) \frac{\eta_{it}(\ell)}{\ell^{1/2-\delta}}
           \frac{H_t(-\delta)}{H_t(\delta)} \left(\frac{q}{4\pi^2}\right)^{-2\delta}
           \\
        & \hskip 100pt
        +  \mathcal{O}_\varepsilon \Big(\big(1+|t|\big)^B q^{-1/2+\varepsilon}\Big).
  \end{split}
\end{equation}
Here we use the bound \eqref{eqn:H<<} for $H_t(s)/H_t(\delta)$.

\subsection{The off-diagonal term}

One can show that
\begin{equation}\label{eqn:cF<<}
  \begin{split}
    \mathcal F & \ll (1+|t|)^B q^{-1/2+\varepsilon}.
  \end{split}
\end{equation}

We won't give the details of the proof of the above estimate,
since a very similar result already appeared in
Kowalski--Michel \cite[Appendix A]{kowalski2000explicit},
and the techniques needed for the proof can be found in Kowalski--Michel \cite[\S2.4]{kowalski2000lower}.

Combining \eqref{eqn:D+F}, \eqref{eqn:cD=}, and \eqref{eqn:cF<<},
completes the proof of theorem \ref{thm:TwistedSecondMoment}.

\end{proof}


\section{Mollification near the critical point}\label{sec:msm}

\subsection{Choosing the mollifier}

We will take the same mollifier as in Kowalski--Michel \cite[Appendix A]{kowalski2000explicit}.
Let $\vartheta$ be a fixed small positive constant,
and $M\leq q^{1/2-2\vartheta}$. Fix $0 < \Upsilon < 1$. Define
\begin{equation*}
  F_{\Upsilon,M}(x) := \left\{\begin{array}{ll}
                   1, & \textrm{if}\ 0\leq x\leq M^{1-\Upsilon}, \\
                   P\left(\frac{\log(M/x)}{\log M}\right), & \textrm{if}\ M^{1-\Upsilon}\leq x\leq M, \\
                   0, & \textrm{if}\ x\geq M,
                 \end{array}\right.
\end{equation*}
with $P(y)=y/\Upsilon$.
We define the mollifier for
$L(s,f)$ by
\begin{equation}\label{eqn:M(s)}
  M(s,f) := \sum_{\ell=1}^{\infty} \frac{x_\ell(s)}{\ell^{1/2}} \lambda_f(\ell),
\end{equation}
where
\begin{equation}\label{eqn:x_l(s)}
  x_\ell(s) := \frac{\mu(\ell)}{\ell^{s-1/2}} \sum_{n=1}^{\infty}\frac{\mu^2(\ell n)F_{\Upsilon,M}(\ell n)}{n^{2s}}.
\end{equation}
We  always write $L(s,f)M(s,f)$ as $LM(s,f)$.

\vskip 3pt
In this section, we will prove the following theorem.
\begin{theorem} \label{thm:MollificationS}
  With the notations as above and $0 < \Upsilon < 1$,  $M\leq q^{1/2-2\vartheta}$, assume that
  $|t|\leq \frac{100\log\log q}{\log q}$ and that
  $-\frac{100}{\log q}\leq \delta \leq \frac{100\log\log q}{\log q}$.
  Then we have
  \[
    \begin{split}
      & \sum_{f\in H_2(q)} \omega_f |LM(1/2+\delta+it,f)|^2
       = 1  + q^{-2\delta}
            \bigg| \frac{M^{-2\delta+2it}-M^{(1-\Upsilon)(-2\delta+2it)}}
            {\Upsilon(-2\delta+2it)\log M} \bigg|^2 \\
      & \hskip 100pt + \frac{M^{-2\delta(1-\Upsilon)}-M^{-2\delta}}{(2\delta\Upsilon\log M)^2}\Big(1-q^{-2\delta}\Big)
      + \mathcal{O}\left(  \frac{\log\log q}{\log q}  M^{-2(1-\Upsilon)\delta} \right).
    \end{split}
  \]
\end{theorem}

\begin{proof}
By \eqref{eqn:HR}, we have
\begin{equation}\label{eqn:|M|^2}
  \begin{split}
     & |M(1/2+\delta+it,f)|^2 \\
       & \hskip 30pt = \sum_{\ell_1=1}^\infty\sum_{\ell_2=1}^\infty \frac{x_{\ell_1}(1/2+\delta+it)x_{\ell_2}(1/2+\delta-it)}
            {\ell_1^{1/2}\ell_2^{1/2}}
            \lambda_{f}(\ell_1)\lambda_f(\ell_2) \\
       & \hskip 30pt = \sum_{\ell_1=1}^\infty\sum_{\ell_2=1}^\infty  \sum_{d=1}^\infty \frac{
       \varepsilon_{q}(d)}{d} \frac{x_{d\ell_1}(1/2+\delta+it)x_{d\ell_2}(1/2+\delta-it)}
            {\ell_1^{1/2}\ell_2^{1/2}}
            \lambda_{f}(\ell_1\ell_2) \\
       & \hskip 30pt = \sum_{\ell_1=1}^\infty\sum_{\ell_2=1}^\infty  \sum_{d=1}^\infty \frac{1}{d} \frac{x_{d\ell_1}(1/2+\delta+it)x_{d\ell_2}(1/2+\delta-it)}
            {\ell_1^{1/2}\ell_2^{1/2}}
            \lambda_{f}(\ell_1\ell_2).
  \end{split}
\end{equation}
Here for the last step we use the fact that the sum over $d$ is actually supported on $[1,M]$
and that $M<q$ and $q$ being prime.
Hence we get
\begin{align*}
    &  \cA\Big(\left\{\omega_f |LM(1/2+\delta+it,f)|^2\right\};\; q\Big)\\
    & \hskip 60pt=  \sum_{\ell_1=1}^\infty\sum_{\ell_2=1}^\infty  \sum_{d=1}^\infty \frac{1}{d} \frac{x_{d\ell_1}(1/2+\delta+it)x_{d\ell_2}(1/2+\delta-it)}
            {\ell_1^{1/2}\ell_2^{1/2}} \\
    &\hskip 180pt \cdot
            \cA\Big(\left\{\omega_f\lambda_f(\ell_1\ell_2)|L(1/2+\delta+it,f)|^2\right\}; \; q\Big).
\end{align*}

There are two cases we need to consider:
\begin{equation*}
  \begin{split}
    & \mathbf{(I)}\left\{\begin{split}
        & \frac{-100}{\log q}\leq \delta \leq \frac{100\log\log q}{\log q}
        \quad\textrm{and}\quad |\delta|\geq \frac{100}{\log q \log\log q}, \\
        & |t| \leq \frac{100\log\log q}{\log q};
      \end{split}\right.
  \end{split}
\end{equation*}
and
\begin{equation*}
  \begin{split}
     \mathbf{(II)}\left\{\begin{split}
        & |\delta|\leq \frac{100}{\log q \log\log q}, \\
        & \frac{1}{100\log q} \leq |t| \leq \frac{100}{\log q}.
      \end{split}\right.
  \end{split}
\end{equation*}
We will focus on the first case $ \mathbf{(I)}$, which we will assume
in this section from now on. Note that the other case
can be handled by combining the ideas of
Conrey and Soundararajan \cite[\S6]{conrey2002real}.
Define
\begin{equation}\label{eqn:nu}
  \nu_{\delta}(r):= \frac{1}{r} \sum_{d|r} \frac{\mu(d)}{d^{1+2\delta}}.
\end{equation}
By theorem \ref{thm:TwistedSecondMoment}, we obtain
\begin{equation}\label{eqn:AtoS}
  \begin{split}
    &\mathcal{A}\Big(\left\{\omega_f |LM(1/2+\delta+it,f)|^2\right\};\; q\Big) \\
    & \hskip 5pt =  \sum_{\ell_1=1}^\infty\sum_{\ell_2=1}^\infty  \sum_{d=1}^\infty \frac{1}{d} \frac{x_{d\ell_1}(1/2+\delta+it)x_{d\ell_2}(1/2+\delta-it)}
            {\ell_1^{1/2}\ell_2^{1/2}}
            \cdot \Bigg[ \zeta_{q}(1+2\delta) \frac{\eta_{it}(\ell_1\ell_2)}{(\ell_1\ell_2)^{1/2+\delta}} \\
    & \hskip 30pt
    + \zeta_{q}(1-2\delta) \frac{\eta_{it}(\ell_1\ell_2)}{(\ell_1\ell_2)^{1/2-\delta}}
    \frac{H_t(-\delta)}{H_t(\delta)} \left(\frac{q}{4\pi^2}\right)^{-2\delta} \Bigg]
     + \mathcal{O}_{\varepsilon}\Big((1+|t|)^B Mq^{-1/2+\ve}\Big)  \\
    & \hskip 5pt =  \mathcal{S}_1 + \mathcal{S}_2 +
    \mathcal{O}_{\varepsilon}\Big((1+|t|)^B q^{-2\vartheta+\varepsilon}\Big),
  \end{split}
\end{equation}
where
\begin{align*}
  \mathcal{S}_1 & := \zeta_{q}(1+2\delta) \sum_{r=1}^{\infty} \nu_{\delta}(r)
   \bigg| \sum_{\ell} \frac{\eta_{it}(\ell)}{\ell^{1+\delta}} x_{r\ell}(1/2+\delta+it) \bigg|^2,  \\
  \mathcal{S}_2 & := \frac{H_t(-\delta)}{H_t(\delta)} \left(\frac{q}{4\pi^2}\right)^{-2\delta}
  \zeta_{q}(1-2\delta) \sum_{r=1}^{\infty} \nu_{-\delta}(r)
   \bigg| \sum_{\ell} \frac{\eta_{it}(\ell)}{\ell^{1-\delta}} x_{r\ell}(1/2+\delta+it) \bigg|^2.
\end{align*}
Here we use the following M\"{o}bius inversion formula
to separate the variables $\ell_1$ and $\ell_2$,
\[
  \eta_{it}(mn) = \sum_{d|(m,n)} \mu(d) \eta_{it}\left(\frac{m}{d}\right)\eta_{it}\left(\frac{n}{d}\right).
\]

\subsection{The treatment of $\mathcal{S}_1$}
\label{subsec:S1}
For $\mathcal{S}_1$, we will follow Conrey and Soundararajan's method in \cite{conrey3000real}
to show
\begin{equation}\label{eqn:S1=}
  \mathcal{S}_1 = 1 + \frac{M^{-2(1-\Upsilon)\delta}-M^{-2\delta}}{(2\delta\Upsilon\log M)^2}
  + \mathcal{O}\left( \frac{\log\log q}{\log q}  M^{-2(1-\Upsilon)\delta} \right).
\end{equation}
One may also use the calculation in
Kowalski--Michel \cite[Apendix D]{kowalski2000explicit} to prove it.
Let $\Psi_{\Upsilon,M}(w)$ 
be the Mellin transform of $F_{\Upsilon,M}(x)$.
We have
\begin{equation}\label{eqn:Psi}
  \Psi_{\Upsilon,M}(w) = \frac{M^{w}-M^{(1-\Upsilon)w}}{\Upsilon w^2 \log M}.
\end{equation}
It has a simple pole at $w=0$ with residue $1$.
By Mellin inversion formula we have
\[
  F_{\Upsilon,M}(x) = \frac{1}{2\pi i} \int\limits_{\alpha-i\infty}^{\alpha+i\infty}
  \Psi_{\Upsilon,M}(w) x^{-w} dw,
\]
for any constant $\alpha>0$.
We have
\begin{align*}
  \mathcal{S}_1 & = \zeta_{q}(1+2\delta) \sum_{r=1}^{\infty} \nu_{\delta}(r)
   \bigg| \sum_{\ell} \frac{\eta_{it}(\ell)}{\ell^{1+\delta}} \frac{\mu(r\ell)}{(r\ell)^{\delta+it}} \sum_{n=1}^{\infty}\frac{\mu^2(r\ell n)F_{\Upsilon,M}(r\ell n)}{n^{1+2\delta+2it}} \bigg|^2 \\
   & = \zeta_{q}(1+2\delta) \sum_{r=1}^{\infty} \frac{\nu_{\delta}(r)}{r^{2\delta}}
   \sum_{\ell_1} \frac{\mu(r\ell_1)\eta_{it}(\ell_1)}{\ell_1^{1+2\delta+it}}  \sum_{n_1=1}^{\infty}\frac{\mu^2(r\ell_1 n_1)F_{\Upsilon,M}(r\ell_1 n_1)}{n_1^{1+2\delta+2it}}  \\
   & \hskip 80pt \cdot \sum_{\ell_2} \frac{\mu(r\ell_2)\eta_{it}(\ell_2)}{\ell_2^{1+2\delta-it}} \sum_{n_2=1}^{\infty}\frac{\mu^2(r\ell_2 n_2)F_{\Upsilon,M}(r\ell_2 n_2)}{n_2^{1+2\delta-2it}} \\
   & = \zeta_{q}(1+2\delta) \frac{1}{(2\pi i)^2} \int\limits_{\alpha-i\infty}^{\alpha+i\infty}\int\limits_{\alpha-i\infty}^{\alpha+i\infty}
       G_1(w_1,w_2,\delta,t)\Psi_{\Upsilon,M}(w_1)\Psi_{\Upsilon,M}(w_2) dw_1 dw_2,
\end{align*}
where
\[
  \begin{split}
    & G_1(w_1,w_2,\delta,t) := \sum_{r=1}^{\infty} \frac{\nu_{\delta}(r)}{r^{2\delta+w_1+w_2}}
        \sum_{\ell_1}\sum_{n_1}\sum_{\ell_2}\sum_{n_2} \\
    & \hskip 120pt \cdot
        \frac{\mu^2(r\ell_1 n_1)\mu(r\ell_1)\eta_{it}(\ell_1)
        \mu^2(r\ell_2 n_2)\mu(r\ell_2)\eta_{it}(\ell_2)}
        {\ell_1^{1+w_1+2\delta+it}n_1^{1+w_1+2\delta+2it}
        \ell_2^{1+w_2+2\delta-it}n_2^{1+w_2+2\delta-2it}}.
  \end{split}
\]
To deal with $G_1(w_1,w_2,\delta,t)$, we will use the following easy lemma.
\begin{lemma}\label{lemma:MultiplicativeFunction1}
  Let $f,g,h$ be three multiplicative functions. Denote
  \[
    F(k) = \sum_{k=rmn} \mu^2(rm)\mu^2(rn) h(r) f(m) g(n).
  \]
  Then $F$ is a multiplicative function.
\end{lemma}
\begin{proof}
  We first note that $F$ is supported on cubic free numbers.
  Let $k=ab^2$, where $a,b$ are square-free numbers.
  We only need to consider the decompostion of $k=rmn$ with
  $b|m$ and $b|n$; otherwise the contribution will be zero
  because of $\mu^2$. Hence
  \[
    F(k)=F(ab^2)=f(b)g(b) H(a),
    \quad \textrm{where} \quad
    H(a) = \sum_{a=rmn} h(r) f(m) g(n).
  \]
  Now it's not hard to check that $F$ is multiplicative.
\end{proof}

We can rearrange the terms in $G_1(w_1,w_2,\delta,t)$ and view it as a
summation over $k=rmn$ where $m=\ell_1n_1$ and $n=\ell_2n_2$. Hence by
lemma \ref{lemma:MultiplicativeFunction1}, we have
\[
  \begin{split}
    & G_1(w_1,w_2,\delta,t) \\
    & \hskip 10pt = \prod_{p} \Bigg( 1 + \frac{1-p^{-1-2\delta}}{p^{1+2\delta+w_1+w_2}}
        -\frac{\eta_{it}(p)}{p^{1+w_1+2\delta+it}}-\frac{\eta_{it}(p)}{p^{1+2\delta+w_2-it}}
        +\frac{1}{p^{1+w_1+2\delta+2it}}+\frac{1}{p^{1+w_2+2\delta-2it}} \\
    & \hskip 80pt + \frac{\eta_{it}(p)^2}{p^{2+w_1+w_2+4\delta}}
        + \frac{1}{p^{2+w_1+w_2+4\delta}}
        -\frac{\eta_{it}(p)}{p^{2+w_1+w_2+4\delta-it}}
        -\frac{\eta_{it}(p)}{p^{2+w_1+w_2+4\delta+it}} \Bigg) \\
    & \hskip 10pt =: \frac{\zeta(1+2\delta+w_1+w_2)}{\zeta(1+w_1+2\delta)\zeta(1+w_2+2\delta)} H_1(w_1,w_2,\delta,t).
  \end{split}
\]
Here the Euler product defining $H_1(w_1,w_2,\delta,t)$ converges absolutely for
\[
  \min\big\{\Re(w_1),\Re(w_2),\Re(w_1+w_2)\big\} > -1/2.
\]

To evaluate the integral in $\mathcal{S}_1$, we shift both contours to the line
$\Re(w_1)=\Re(w_2)=(2A+1)/\log q$, and truncate the $w_2$
integral at $|\Im(w_2)|\leq q$ with error $\mathcal{O}(q^{-1+\varepsilon})$.
Then we shift the $w_1$ integral to the contour $\mathcal{C}_1$ given by
\[
  \mathcal{C}_1 := \left\{ w_1 \; \Big| \;
  \Re(w_1) = -2\delta - \frac{c}{\log^{3/4}\big(2+|\Im(w_1)|\big)} \right\},
\]
where $c>0$ is a suitable constant such that $\zeta(s)$ has no zero in the region
$\Re(s)\geq 1-c\log^{-3/4}\big(2+|\Im(s)|\big)$ and
$c\log^{-3/4}(2)<1/3$.
If $q$ is sufficiently large, we encounter poles at $w_1=0$ and $w_1=-2\delta-w_2$.
The first pole at $w_1=0$ yields a residue
\begin{equation}\label{eqn:Res0S1}
  \begin{split}
    & \frac{1}{\zeta(1+2\delta)} \frac{1}{2\pi i}
      \int\limits_{(2A+1)/\log q -iq}^{(2A+1)/\log q +iq}
      H_1(0,w_2,\delta,t)  \Psi_{\Upsilon,M}(w_2) dw_2  \\
    & \hskip 120pt = \frac{1}{\zeta(1+2\delta)} \bigg( 1 + \mathcal{O}_{\varepsilon} \left(M^{-(1-\Upsilon)/2+\varepsilon}\right)\bigg),
  \end{split}
\end{equation}
by moving the $w_2$-contour to $\mathcal{C}_{21}$ where
\[
  \begin{split}
    \mathcal{C}_{21} := \left[\frac{2A+1}{\log q}-iq,-\frac{1}{2}-iq\right] \cup
    \left[-\frac{1}{2}-iq,-\frac{1}{2}+iq\right] \cup \left[-\frac{1}{2}+iq,\frac{2A+1}{\log q} +iq\right],
  \end{split}
\]
and estimating the integration trivially.

The pole at $w_1=-2\delta-w_2$ leaves the residue
\begin{equation}\label{eqn:Res1S1}
  \frac{1}{2\pi i} \int\limits_{(2A+1)/\log q -iq}^{(2A+1)/\log q +iq}
     \frac{H_1(-2\delta-w_2,w_2,\delta,t) }{\zeta(1-w_2)\zeta(1+2\delta+w_2)} \Psi_{\Upsilon,M}(-2\delta-w_2) \Psi_{\Upsilon,M}(w_2) dw_2.
\end{equation}
Note that by \eqref{eqn:Psi} we have
\begin{equation}\label{eqn:Psi*Psi}
  \Psi_{\Upsilon,M}(-2\delta-w_2) \Psi_{\Upsilon,M}(w_2)
  = \frac{M^{-2\delta}+M^{-2\delta(1-\Upsilon)}-M^{-\Upsilon w_2-2\delta}-M^{\Upsilon w_2-2\delta(1-\Upsilon)}}{w_2^2(2\delta+w_2)^2(\Upsilon\log M)^2}.
\end{equation}
We can see that the contributions from the first three terms on the right hand side of the above identity, that is, $\frac{M^{-2\delta}+M^{-2\delta(1-\Upsilon)}-M^{-\Upsilon w_2-2\delta}}{w_2^2(2\delta+w_2)^2(\Upsilon\log M)^2}$, are small.
Indeed, by shifting the integration to the contour $\mathcal{C}_{22}$ where
\[
  \begin{split}
    & \mathcal{C}_{22} := \left[\frac{2A+1}{\log q}-iq,\frac{2A+1}{\log q}-i\right]\cup
      \left[\frac{2A+1}{\log q}-i,\frac{1}{4}-i\right] \cup
      \left[\frac{1}{4}-i,\frac{1}{4}+i\right] \\
    & \hskip80pt \cup
      \left[\frac{1}{4}+i,\frac{2A+1}{\log q}+i\right] \cup
      \left[\frac{2A+1}{\log q}+i,\frac{2A+1}{\log q}+iq\right],
  \end{split}
\]
and using the standard bounds for the Riemann zeta function,
one can show that the contributions from the first three terms are bounded by
\[
  \mathcal{O}\left( \frac{M^{-2\delta(1-\Upsilon)}}{(\log M)^2}\right).
\]
Now we consider the contribution from $\frac{-M^{\Upsilon w_2-2\delta(1-\Upsilon)}}{w_2^2(2\delta+w_2)^2(\Upsilon\log M)^2}$
to \eqref{eqn:Res1S1}.
We move the contour of integration to the left $\mathcal{C}_{23}$ where
\[
  \begin{split}
    & \mathcal{C}_{23} := \left[\frac{2A+1}{\log q}-iq,-2\delta-\frac{c}{\log^{3/4} q}-iq\right]
    \\
    & \hskip 5pt \cup
      \left[-2\delta-\frac{c}{\log^{3/4} q}-iq,-2\delta-\frac{c}{\log^{3/4} q}+iq\right] \cup
      \left[-2\delta-\frac{c}{\log^{3/4} q}+iq,\frac{2A+1}{\log q}+iq\right],
  \end{split}
\]
for the same constant $c$ in the definition of $\mathcal{C}_1$,
and pick up contributions from the poles at $w_2=0$ and $w_2=-2\delta$.
Similarly, we find the integration on $\mathcal{C}_{23}$ can be bounded by
\[
  \mathcal{O}\Big( M^{-2\delta(1-\Upsilon)} e^{-c'\log^{1/4}q} \Big),
  \quad \mbox{for some $c'>0$.} 
\]
The pole at $w_2=0$ contributes
\begin{equation}\label{eqn:Res11S1}
  \frac{1}{\zeta(1+2\delta)}\frac{M^{-2\delta(1-\Upsilon)}}{(2\delta \Upsilon \log M)^2} H_1(-2\delta,0,\delta,t).
\end{equation}
Lastly the pole at $w_2=-2\delta$ leaves the residue
\begin{equation}\label{eqn:Res12S1}
  -\frac{1}{\zeta(1+2\delta)}\frac{M^{-2\delta}}{(2\delta \Upsilon \log M)^2} H_1(0,-2\delta,\delta,t).
\end{equation}

The remaining integral, for $w_1$ on $\mathcal{C}_1$, is bounded by
\[
  \mathcal{O}\Big( M^{-2\delta(1-\Upsilon)} e^{-c'\log^{1/4}q} \Big),
  \quad \mbox{for some $c'>0$.}
\]
The proof that $M\leq q^{1/2-2\vartheta}$ now follows upon  combining \eqref{eqn:Res0S1}, \eqref{eqn:Res11S1}, and \eqref{eqn:Res12S1},
and noting that $H_1(0,0,\delta,t)=1$.

\subsection{The treatment of $\mathcal{S}_2$}

We now consider $\mathcal{S}_2$.
Following a similar argument as for $\mathcal{S}_1$, we have
\[
  \begin{split}
    \mathcal{S}_2 & = \frac{H_t(-\delta)}{H_t(\delta)} \left(\frac{q}{4\pi^2}\right)^{-2\delta}
    \zeta_{q}(1-2\delta) \sum_{r=1}^{\infty} \nu_{-\delta}(r)
    \frac{\mu^2(r)}{r^{2\delta}} \sum_{\ell_1=1}^{\infty}\sum_{n_1=1}^{\infty} \frac{\eta_{it}(\ell_1)\mu(\ell_1)\mu^2(r\ell_1 n_1)}{\ell_1^{1+it}n_1^{1+2\delta+2it}}  \\
    & \hskip 80pt \cdot \sum_{\ell_2=1}^{\infty}\sum_{n_2=1}^{\infty} \frac{\eta_{it}(\ell_2)\mu(\ell_2)\mu^2(r\ell_2 n_2)}{\ell_2^{1-it}n_2^{1+2\delta-2it}} F_{\Upsilon,M}(r\ell_1 n_1) F_{\Upsilon,M}(r\ell_2 n_2) \\
    & =  \frac{H_t(-\delta)}{H_t(\delta)} \left(\frac{q}{4\pi^2}\right)^{-2\delta}
       \zeta_{q}(1-2\delta) \\
    & \hskip 80pt \cdot
       \frac{1}{(2\pi i)^2} \int\limits_{\alpha-i\infty}^{\alpha+i\infty}\int\limits_{\alpha-i\infty}^{\alpha+i\infty}
       G_2(w_1,w_2,\delta,t)\Psi_{\Upsilon,M}(w_1)\Psi_{\Upsilon,M}(w_2) dw_1 dw_2,
  \end{split}
\]
where
\[
  \begin{split}
    & G_2(w_1,w_2,\delta,t) \\
    & \hskip 5pt := \sum_{r} \frac{\mu^2(r)\nu_{-\delta}(r)}{r^{1+2\delta+w_1+w_2}}
     \sum_{\ell_1}\sum_{n_1}\sum_{\ell_2}\sum_{n_2}
     \frac{\mu^2(r\ell_1n_1)\mu(\ell_1)\eta_{it}(\ell_1)
     \mu^2(r\ell_2n_2)\mu(\ell_2)\eta_{it}(\ell_2)}
     {\ell_1^{1+w_1+it}n_1^{1+w_1+2\delta+2it}
     \ell_2^{1+w_2-it}n_1^{1+w_2+2\delta-2it}} \\
    & \hskip 5pt = \prod_{p} \Bigg( 1 + \frac{1-p^{-1+2\delta}}{p^{1+2\delta+w_1+w_2}}
        -\frac{\eta_{it}(p)}{p^{1+w_1+it}}-\frac{\eta_{it}(p)}{p^{1+w_2-it}}
        +\frac{1}{p^{1+w_1+2\delta+2it}}+\frac{1}{p^{1+w_2+2\delta-2it}} \\
    & \hskip 45pt + \frac{\eta_{it}(p)^2}{p^{2+w_1+w_2}}
        + \frac{1}{p^{2+w_1+w_2+4\delta}}
        -\frac{\eta_{it}(p)}{p^{2+w_1+w_2+2\delta-it}}
        -\frac{\eta_{it}(p)}{p^{2+w_1+w_2+2\delta+it}} \Bigg) \\
    & \hskip 5pt =: \frac{\zeta(1+2\delta+w_1+w_2)\zeta(1+w_1+2\delta+2it)\zeta(1+w_2+2\delta-2it)}
    {\zeta(1+w_1)\zeta(1+w_1+2it)\zeta(1+w_2)\zeta(1+w_2-2it)} H_2(w_1,w_2,\delta,t).
  \end{split}
\]
Here the Euler product defining $H_2(w_1,w_2,\delta,t)$ converges absolutely for
\[
  \min\big\{\Re(w_1),\Re(w_2),\Re(w_1+w_2)\big\} > -1/2.
\]

To estimate the integral in $\mathcal{S}_2$, we again shift the $w_1$ and $w_2$ contours to
the lines $\Re(w_1)=\Re(w_2)=\frac{2A+1}{\log q}$.
We now truncate the $w_2$ contour at $\Im(w_2)\leq q$,
and shift the $w_1$ contour to $\mathcal{C}_1'$ given by
\[
  \mathcal{C}_1' := \left\{ w_1 \;\Bigg |\;
  \Re(w_1)=-\frac{c}{\log^{3/4}\big(2+|\Im(w_1+2it)|\big)} \right\},
\]
for the same constant $c$ in the definition of $\mathcal{C}_1$.
If $q$ is sufficiently large, we encounter poles at $w_1=-2\delta-2it$ and $w_1=-2\delta-w_2$.
The first pole at $w_1=-2\delta-2it$ yields a residue
\begin{equation}\label{eqn:Res0S2}
  \begin{split}
    & \frac{1}{\zeta(1-2\delta)\zeta(1-2\delta-2it)} \cdot \frac{1}{2\pi i}
      \int\limits_{(2A+1)/\log q -iq}^{(2A+1)/\log q +iq} \frac{\zeta(1+w_2+2\delta-2it)}{\zeta(1+w_2)} \\
    & \hskip 125pt \cdot  H_2(0,w_2,\delta,t)  \Psi_{\Upsilon,M}(-2\delta-2it)
    \Psi_{\Upsilon,M}(w_2) dw_2.
  \end{split}
\end{equation}
We now move the $w_2$-contour to $\mathcal{C}_{21}'$ where
\[
  \begin{split}
    & \mathcal{C}_{21}' := \left[\frac{2A+1}{\log q}-iq,\alpha_0-iq\right]\cup
      \left[\alpha_0-iq,\alpha_0+iq\right] \cup
      \left[\alpha_0+iq,\frac{2A+1}{\log q}+iq\right]
  \end{split}
\]
with $\alpha_0=\max\left\{-\frac{c}{\log^{3/4} q},-2\delta-\frac{c}{\log^{3/4} q}\right\}$,
pick up the pole at $w_2=-2\delta+2it$ when $q$ is large enough,
and estimate the remaining integration trivially.
By doing so, we find that \eqref{eqn:Res0S2} is equal to
\begin{equation}\label{eqn:Res00S2}
  \begin{split}
    & \frac{H_2(0,-2\delta+2it,\delta,t)}{\zeta(1-2\delta)\zeta(1-2\delta-2it)\zeta(1-2\delta+2it)}
      \Psi_{\Upsilon,M}(-2\delta-2it)  \Psi_{\Upsilon,M}(-2\delta+2it) \\
    & \hskip 240pt + \mathcal{O}\left( M^{-2\delta(1-\Upsilon)} e^{-c'\log^{1/4}q} \right),
  \end{split}
\end{equation}
for some constant $c'>0$ depending on $c$ and $\Upsilon$.

The pole at $w_1=-2\delta-w_2$ leaves the residue
\begin{equation}\label{eqn:Res1S2}
  \begin{split}
    & \frac{1}{2\pi i} \int\limits_{(2A+1)/\log q -iq}^{(2A+1)/\log q +iq}
     \frac{\zeta(1-w_2+2it)\zeta(1+w_2+2\delta-2it)H_2(-2\delta-w_2,w_2,\delta,t) }{\zeta(1-2\delta-w_2)\zeta(1-w_2-2\delta+2it)\zeta(1+w_2)\zeta(1+w_2-2it)}\\
    & \hskip 240pt \cdot \Psi_{\Upsilon,M}(-2\delta-w_2) \Psi_{\Upsilon,M}(w_2) dw_2.
  \end{split}
\end{equation}
Note that we have \eqref{eqn:Psi*Psi}.
We can see that the contributions from the first three terms there are small again.
Indeed, by the same method as in the treatment of $\mathcal{S}_1$ we see that
the contributions from the first three terms are
\[
  \mathcal{O}\left( \frac{M^{-2\delta(1-\Upsilon)}}{(\log M)^2}\right).
\]
Now we consider the contribution from $\frac{-M^{\Upsilon w_2-2\delta(1-\Upsilon)}}{w_2^2(2\delta+w_2)^2(\Upsilon\log M)^2}$
to \eqref{eqn:Res1S2}.
We move the contour of integration to the left $\mathcal{C}_{23}'$ where
\[
  \begin{split}
    & \mathcal{C}_{23}' := \left[\frac{2A+1}{\log q}-iq,-\frac{c}{\log^{3/4} q}-iq\right]\cup
      \left[-\frac{c}{\log^{3/4} q}-iq,-\frac{c}{\log^{3/4} q}+iq\right] \\
    & \hskip 50pt \cup
      \left[-\frac{c}{\log^{3/4} q}+iq,\frac{2A+1}{\log q}+iq\right],
  \end{split}
\]
for the same constant $c$ in the definition of $\mathcal{C}_1$,
and pick up contributions from the poles at $w_2=0$ and $w_2=-2\delta$.
Similarly, we find the integration on $\mathcal{C}_{23}'$ can be bounded by
\[
  \mathcal{O}\Big( M^{-2\delta(1-\Upsilon)} e^{-c'\log^{1/4}q} \Big),
  \quad \mbox{for some $c'>0$.}
\]
The pole at $w_2=0$ contributes
\begin{equation}\label{eqn:Res11S2}
  -\frac{1}{\zeta(1-2\delta)}\frac{M^{-2\delta(1-\Upsilon)}}{(2\delta \Upsilon \log M)^2} H_2(-2\delta,0,\delta,t).
\end{equation}
Lastly the pole at $w_2=-2\delta$ leaves the residue
\begin{equation}\label{eqn:Res12S2}
  \frac{1}{\zeta(1+2\delta)}\frac{M^{-2\delta}}{(2\delta \Upsilon \log M)^2} H_2(0,-2\delta,\delta,t).
\end{equation}

The remaining integral, for $w_1$ on $\mathcal{C}_{1}'$, is bounded by
\[
  \mathcal{O}\Big( M^{-2\delta(1-\Upsilon)} e^{-c'\log^{1/4}q} \Big),
  \quad \mbox{for some $c'>0$.}
\]
Combining \eqref{eqn:Res00S2}, \eqref{eqn:Res11S2}, and \eqref{eqn:Res12S2},
and noting that $H_2(0,0,0,0)=1$., we obtain
\begin{equation}\label{eqn:S2=}
  \begin{split}
     \mathcal{S}_2 & = q^{-2\delta}
            \bigg| \frac{M^{-2\delta+2it}-M^{(1-\Upsilon)(-2\delta+2it)}}
            {\Upsilon(-2\delta+2it)\log M} \bigg|^2
            \\
      & \hskip 100pt   - q^{-2\delta} \frac{M^{-2\delta(1-\Upsilon)}-M^{-2\delta}}{(2\delta\Upsilon\log M)^2}
      + \mathcal{O}\left(  \frac{\log\log q}{\log q}  M^{-2(1-\Upsilon)\delta} \right).
  \end{split}
\end{equation}
Finally, by \eqref{eqn:AtoS}, \eqref{eqn:S1=}, and \eqref{eqn:S2=},
 this completes  the proof of theorem \ref{thm:MollificationS}.

\end{proof}


\section{Mollification away from the critical line} \label{sec:mollification_away}

In this section, we will prove  theorem  \ref{thm:mollificationL_h} which is not  needed for the proof of our main theorems. However, the results of this section will serve as a prototype for a similar result (without the harmonic weights) which is actually needed in \S\ref{subsec:SM_L}.

\begin{theorem}\label{thm:mollificationL_h}
  With notations as in \S\ref{sec:msm}, assume $s=\frac12+\delta+it$ where
  $$\frac{1}{\log q} \;\leq\; \delta \;\leq\; \frac12+\frac{10 \log\log q}{(1-\Upsilon)\log q}.$$ Then there exists
  an absolute constant $B>0$ such that for any $0<a<2(1-\Upsilon)$ we have
  \[
    \sum_{f\in H_2(q)} \omega_f \big(|LM(s,f)|^2 - 1\big) \,
    \ll_{\Upsilon,\vartheta,a,B} \; (1+|t|)^B M^{-a\delta}.
  \]
\end{theorem}

\begin{proof}[Proof of theorem \ref{thm:mollificationL_h}]
The proof follows easily from lemmas \ref{lemma:SecondMomentL_h}, \ref{lemma:FirstMomentL_h} which are stated and proved below.
\end{proof}

\begin{lemma}\label{lemma:SecondMomentL_h}
  With notations as above, assume $s=1/2+\delta+it$ with
  $\delta\geq 1/(2\log q)$. Then for any $0<a<2(1-\Upsilon)$, we have
  \[
    \sum_{f\in H_2(q)} \omega_f |LM(s,f)-1|^2 \, \ll_{\Upsilon,\vartheta,a,B} \; (1+|t|)^B M^{-a\delta}.
  \]
\end{lemma}

\begin{proof}
  See Kowalski--Michel \cite[Corollary 1]{kowalski2000explicit}.
\end{proof}

\begin{lemma}\label{lemma:FirstMomentL_h}
 With notations as above, assume that
  $\frac{1}{\log q}\leq \delta \leq 1/2+\frac{10 \log\log q}{(1-\Upsilon)\log q}$.
  Then for any $0<a\leq 2(1-\Upsilon)$, we have
  \[
    \sum_{f\in H_2(q)} \omega_f \big(LM(s,f)-1\big) \,
    \ll_{\Upsilon,\vartheta,B} \; (1+|t|)^B M^{-a\delta}.
  \]
\end{lemma}

\begin{remark}
  This is an improvement of lemma 7 in Kowalski--Michel \cite{kowalski2000explicit},
  where they proved a similar bound for any $0<a<1-\Upsilon$.
  We will prove the lemma by using the method of Conrey and Soundararajan \cite{conrey3000real}
  (see also \cite[Lemma 6.6]{goldfeld2016super}).
  The main idea here is to use the approximate functional equation and the Petersson trace formula
  instead of obtaining the bound by the second moment and the Cauchy--Schwarz inequality.
\end{remark}

\begin{proof}[Proof of lemma \ref{lemma:FirstMomentL_h}]
  In the region $\Re(s)>1$, by \eqref{eqn:M(s)} we may write
  \[
    LM(s,f) = \sum_{n=1}^{\infty} \frac{1}{n^{s}} \left( \sum_{abc^2=n} \lambda_f(a)\lambda_f(b)\mu(b)\mu(bc)^2 F_{\Upsilon,M}(bc) \right).
  \]
  Using the Hecke relations and the fact that we may assume $b<q$, we see that
  \[
    \begin{split}
      & \sum_{abc^2=n} \lambda_f(a)\lambda_f(b)\mu(b)\mu(bc)^2 F_{\Upsilon,M}(bc)
      = \sum_{abc^2=n} \sum_{d|(a,b)}\lambda_f\left(\frac{ab}{d^2}\right)\mu(b)\mu(bc)^2 F_{\Upsilon,M}(bc),
    \end{split}
  \]
  and setting $a=\alpha d$, $b=\beta d$, and $g=cd$, this becomes
  \[
    \begin{split}
      & \sum_{g^2|n} \lambda_f\left(\frac{n}{g^2}\right)
      \sum_{\alpha\beta=n/g^2}\mu(\beta g)^2 F_{\Upsilon,M}(\beta g) \sum_{cd=g}\mu(\beta d)
      = \lambda_f(n) \sum_{\alpha\beta=n} \mu(\beta) F_{\Upsilon,M}(\beta),
    \end{split}
  \]
  since the terms with $g>1$ are easily seen to disappear. Thus
  \[
    LM(s,f) = \sum_{n=1}^{\infty} \frac{\lambda_f(n)}{n^{s}} c(n),
    \quad \textrm{where}\quad
    c(n) = \sum_{d|n} \mu(d) F_{\Upsilon,M}(d).
  \]
  We have $c(1)=1$; for $1<n\leq M^{1-\Upsilon}$, we have $c(n)= \sum_{d|n}\mu(d)=0$; and for $n>M^{1-\Upsilon}$,  we have $|c(n)|\leq \tau(n)$.

  We will first handle the case $\Re(s)= 1/2+\delta_0$ where
  $\delta_0 = 1/2+\frac{10 \log\log q}{(1-\Upsilon)\log q}$.
  Put $B(s,f):=LM(s,f)-1$. We consider
  \[
    \frac{1}{2\pi i} \int\limits_{3-i\infty}^{3+i\infty}\Gamma(w) B(w+s,f)X^w  dw,
  \]
  where $X=q^{1-\vartheta}$.
  We shift the line of integration to $\Re(w)=-\delta_0+\delta_1$,
  where $\delta_1=\frac{1}{\log q}$.
  The pole at $w=0$ gives $B(s,f)$, and so we conclude that
  \[
    \begin{split}
      B(s,f) & = \frac{1}{2\pi i} \int\limits_{3-i\infty}^{3+i\infty} \Gamma(w)B(w+s,f)X^w  dw
       - \frac{1}{2\pi i} \int\limits_{-\delta_0+\delta_1-i\infty}^{-\delta_0+\delta_1+i\infty} \Gamma(w)B(w+s,f)X^w dw \\
      & =: T_1(s,f) - T_2(s,f).
    \end{split}
  \]

  We first estimate the contribution of the $T_2(s,f)$ terms.
  By Cauchy's inequality and lemma \ref{lemma:SecondMoment},
  for some constant $B>0$ we have
  \begin{align*}
      &  \mathcal{A}\Big( \big\{ \omega_f |T_2(s,f)| \big\};q \Big)
     \ll  X^{-\delta_0+\delta_1}
     \mathcal{A}\Bigg( \bigg\{ \omega_f \int_{(-\delta_0+\delta_1)}
     |\Gamma(w)| \cdot |B(w+s,f)|\; |dw| \bigg\}; \;q \Bigg)  \\
     & \hskip 75pt \ll   (1+|t|)^B X^{-\delta_0+\delta_1}
     \ll (1+|t|)^B q^{-(1-2\vartheta)\delta_0}
     \ll (1+|t|)^B M^{-2(1-\Upsilon)\delta_0}.
  \end{align*}
  It remains now to estimate the $T_1$ contribution.
  Since $\frac{1}{2\pi i}\int_{(\alpha)}\Gamma(w)(X/n)^w dw = e^{-n/X}$,
  we see that
  \[
    \begin{split}
      T_1(s,f) & = \sum_{n=2}^{\infty} \frac{\lambda_f(n)c(n)}{n^s} e^{-n/X} = \sum_{M^{1-\Upsilon}<n\leq X(\log q)^2} \frac{\lambda_f(n)c(n)}{n^s} e^{-n/X} + \mathcal{O}\Big(q^{-B}\Big).
    \end{split}
  \]
  In order to bound $ \mathcal{A}\big( \big\{ \omega_f T_1(s,f) \big\};q \big)$,
  for $M^{1-\Upsilon}<n\leq X(\log q)^2$, we consider
  \[
    \mathcal{A}\big( \big\{ \omega_f \lambda_f(n) \big\};q \big)
    =  \sum_{f\in H_2(q)}  \omega_f \cdot \lambda_f(n).
  \]
  By lemma \ref{lemma:PTF}, we arrive at
  \[
    \begin{split}
         & - 2\pi
         \sum_{c\,\equiv\,0\,(\mod \hskip -3ptq)} \frac{S(1,n;c)}{c} J_{1}\left(\frac{4\pi\sqrt{n}}{c}\right)
         \ll \frac{\sqrt{n}}{q^{3/2-\varepsilon}} \sum_{r\geq1} r^{-3/2}
         \ll \frac{\sqrt{n}}{q^{3/2-\varepsilon}}.
    \end{split}
  \]
  Thus we see that if $n\leq X(\log q)^2$ then
  \[
    \mathcal{A}\big( \big\{ \omega_f \lambda_f(n) \big\}; \, q \big)
      \ll   \frac{\sqrt{n}}{q^{3/2-\varepsilon}}
      \ll   q^{-1-\varepsilon}.
  \]
  Hence
  \[
    \mathcal{A}\big( \big \{ \omega_f T_1(s,f) \big\}; \, q \big)
     \ll q^{-1},
  \]
  and
  \[
    \mathcal{A}\big( \big\{ \omega_f (LM(1/2+\delta_0+it,f)-1) \big\};q \big)
    \ll_{B,a}  (1+|t|)^B M^{-2(1-\Upsilon)\delta_0}.
  \]
  This completes the proof for $\Re(s)=\frac12+\delta_0 = 1+\frac{10 \log\log q}{(1-\Upsilon)\log q}$.
  Furthermore, the result for $\frac{1}{\log q}\leq \delta \leq \frac12+\frac{10 \log\log q}{(1-\Upsilon)\log q}$
  now follows from lemma \ref{lemma:SecondMomentL_h} and the convexity argument.
\end{proof}


\section{Removing the harmonic weights}\label{sec:removing_weights}

\subsection{Near the critical point}

To get the natural average from the harmonic average,
we will use the method of \cite{kowalski1999analytic}.
For $x\geq1$, let
\begin{equation}\label{eqn:w(x)}
  \omega_f(x) := \sum_{dl^2\leq x}\frac{\lambda_f(d^2)}{dl^2}.
\end{equation}
We will use the following lemma to remove the harmonic weights.
\begin{lemma}\label{lemma:Kowalski-Michel}
  Let $\{\alpha_f\}$ be complex numbers satisfying the conditions
  \begin{equation}\label{eqn:condition1}
    \mathcal{A}\big(\{\omega_f \alpha_f\};q\big) \ll (\log q)^C, \quad
    \mbox{(for some absolute constant $C>0$)},
  \end{equation}
  and
  \begin{equation}\label{eqn:condition2}
    \underset{f\in H_2(q)}{\max} |\omega_f \alpha_f| \ll q^{-\varrho}, \quad
    \mbox{(for some $\varrho>0$)},
  \end{equation}
  as the level $q$ (prime) tends to infinity. Let $x=q^\kappa$ for some $0<\kappa<1$.
  There exists an absolute constant $\varsigma=\varsigma(\kappa,\varrho)>0$ such that
  \[
    \frac{\mathcal{A}\big(\{\alpha_f\};q\big)}{\mathcal{A}(q)}
    = \frac{1}{\zeta(2)}\mathcal{A}\big(\{\omega_f \cdot\omega_f(x) \alpha_f\};q\big)
    + \mathcal{O}(q^{-\varsigma}).
  \]
\end{lemma}

\begin{proof}
  See \cite[Proposition 2]{kowalski1999analytic}.
\end{proof}

Using the above lemma and  similar calculations as in \S \ref{sec:msm},
we can obtain the following theorem.

\begin{theorem} \label{thm:MollificationS2}
  Let $q,M,\Upsilon$ be the same as in theorem \ref{thm:MollificationS}.
  Let $x=q^\kappa$ where $0<\kappa<1$ is a small constant.
  Assume that $|t|\leq \frac{100\log\log q}{\log q}$ and $-\frac{100}{\log q}\leq \delta \leq \frac{100\log\log q}{\log q}$.
  Then we have
  \begin{equation}\label{eqn:SecondMoment-w(x)}
    \begin{split}
      & \frac{1}{\zeta(2)} \sum_{f\in H_2(q)} \omega_f \cdot \omega_f(x) \cdot |LM(1/2+\delta+it,f)|^2  \\
      & \hskip 20pt
       = 1  + q^{-2\delta}
            \bigg| \frac{M^{-2\delta+2it}-M^{(1-\Upsilon)(-2\delta+2it)}}
            {\Upsilon(-2\delta+2it)\log M} \bigg|^2   \\
      & \hskip 80pt
      + \frac{M^{-2\delta(1-\Upsilon)}-M^{-2\delta}}{(2\delta\Upsilon\log M)^2}\Big(1-q^{-2\delta}\Big)
      + \mathcal{O}\left(  \frac{\log\log q}{\log q}  M^{-2(1-\Upsilon)\delta} \right).
    \end{split}
  \end{equation}
  Furthermore, we have
  \begin{equation}\label{eqn:SecondMoment-n}
    \begin{split}
      & \frac{\mathcal{A}\big(\{|LM(1/2+\delta+it,f)|^2\};q\big)}{\mathcal{A}(q)}
       = 1  + q^{-2\delta}
            \bigg| \frac{M^{-2\delta+2it}-M^{(1-\Upsilon)(-2\delta+2it)}}
            {\Upsilon(-2\delta+2it)\log M} \bigg|^2 \\
      & \hskip 80pt + \frac{M^{-2\delta(1-\Upsilon)}-M^{-2\delta}}{(2\delta\Upsilon\log M)^2}\Big(1-q^{-2\delta}\Big)
      + \mathcal{O}\left(  \frac{\log\log q}{\log q}  M^{-2(1-\Upsilon)\delta} \right).
    \end{split}
  \end{equation}
\end{theorem}

\begin{proof}
  We will follow \cite[\S6]{hough2012zero} closely.
  By \eqref{eqn:|M|^2}, \eqref{eqn:w(x)}, and the Hecke relations, we have
  \[
    \begin{split}
      & \omega_f(x) |M(1/2+\delta+it,f)|^2 = \sum_{dl^2\leq x}\frac{1}{dl^2} \sum_{\ell_1=1}^\infty\sum_{\ell_2=1}^\infty
      \sum_{g=1}^\infty \frac{1}{g}  \\
      & \hskip 80pt \cdot \frac{x_{g\ell_1}(1/2+\delta+it)x_{g\ell_2}(1/2+\delta-it)}
            {\ell_1^{1/2}\ell_2^{1/2}}
            \sum_{h|(d^2,\ell_1\ell_2)}\lambda_f\left(\frac{d^2\ell_1\ell_2}{h^2}\right).
    \end{split}
  \]
  Write $h=h_1h_2^2$ where $h_1$ and $h_2$ are square-free.
  Clearly $h_2|(\ell_1,\ell_2)$ and $h_2|d$.
  Also $h_1|(d,\ell_1\ell_2)$. Shifting the orders of summations, and then
  making changes of variables ($d\rightarrow dh_1h_2,\ \ell_1\rightarrow \ell_1h_2$,
  and $\ell_2\rightarrow \ell_2h_2$), we obtain
  \[
    \begin{split}
       & \frac{1}{\zeta(2)} \sum_{f\in H_2(q)} \omega_f \cdot \omega_f(x) |LM(1/2+\delta+it,f)|^2
       = \frac{1}{\zeta(2)}  \sum_{dl^2\leq x}\frac{1}{dl^2} \sum_{\ell_1=1}^\infty\sum_{\ell_2=1}^\infty
      \sum_{g=1}^\infty \frac{1}{g}
       \sum_{h|(d^2,\ell_1\ell_2)}  \\
      & \hskip 25pt \cdot \frac{x_{g\ell_1}(1/2+\delta+it)x_{g\ell_2}(1/2+\delta-it)}
            {\ell_1^{1/2}\ell_2^{1/2}}
           \sum_{f\in H_2(q)} \omega_f \cdot \lambda_f\left(\frac{d^2\ell_1\ell_2}{h^2}\right) \big |L(1/2+\delta+it)\big |^2 \\
      & \hskip 15pt = \frac{1}{\zeta(2)}  \sum_{l}\frac{1}{l^2} \sum_{\ell_1=1}^\infty\sum_{\ell_2=1}^\infty
      \sum_{g=1}^\infty \frac{1}{g} \underset{h_2}{{\sum}^{\flat}} \frac{1}{h_2^2}
      \frac{x_{g\ell_1h_2}(1/2+\delta+it)x_{g\ell_2h_2}(1/2+\delta-it)}{\ell_1^{1/2}\ell_2^{1/2}} \\
      & \hskip 120pt \cdot
            \underset{h_1|\ell_1\ell_2}{{\sum}^{\flat}} \frac{1}{h_1} \sum_{d<x/l^2h_1h_2} \frac{1}{d} \sum_{f\in H_2(q)} \omega_f \cdot \lambda_f\left(d^2\ell_1\ell_2\right) \big |L(1/2+\delta+it)\big |^2.
    \end{split}
  \]
  Now introducing our expression for the twisted second moment, i.e., theorem \ref{thm:TwistedSecondMoment}, we obtain
  \[
    \begin{split}
       & \frac{1}{\zeta(2)} \sum_{f\in H_2(q)} \omega_f \cdot \omega_f(x) |LM(1/2+\delta+it,f)|^2 \\
      & \hskip 30pt
       = \frac{1}{\zeta(2)}  \sum_{l}\frac{1}{l^2} \sum_{\ell_1=1}^\infty\sum_{\ell_2=1}^\infty
      \sum_{g=1}^\infty \frac{1}{g} \underset{h_2}{{\sum}^{\flat}} \frac{1}{h_2^2} \frac{x_{g\ell_1h_2}(1/2+\delta+it)x_{g\ell_2h_2}(1/2+\delta-it)}{\ell_1^{1/2}\ell_2^{1/2}}
       \\
      & \hskip 60pt \cdot \underset{h_1|\ell_1\ell_2}{{\sum}^{\flat}} \frac{1}{h_1} \sum_{d<x/l^2h_1h_2} \frac{1}{d}
           \Bigg\{ \zeta_q(1+2\delta) \frac{\eta_{it}(d^2\ell_1\ell_2)}{(d^2\ell_1\ell_2)^{1/2+\delta}}
             \\
      & \hskip 100pt
       + \zeta_q(1-2\delta) \frac{\eta_{it}(d^2\ell_1\ell_2)}{(d^2\ell_1\ell_2)^{1/2-\delta}}
                \frac{H_t(-\delta)}{H_t(\delta)} \left(\frac{q}{4\pi^2}\right)^{-2\delta}
             +  \mathcal{O}_{\varepsilon}\Big(q^{-1/2+\ve}\Big) \Bigg\}.
    \end{split}
  \]
  The error term above contributes $\mathcal{O}_{\varepsilon}\left(Mq^{-1/2+\varepsilon}\right)
  =\mathcal{O}_{\varepsilon}\left(q^{-2\vartheta+\varepsilon}\right)$.
  Note that we may remove the restriction on the sum over $d$ with error
  $\ll x^{-1/2+\varepsilon}$. Thus
  \begin{equation}\label{eqn:T12}
    \begin{split}
       & \frac{1}{\zeta(2)} \sum_{f\in H_2(q)} \omega_f \cdot \omega_f(x)\big |LM(1/2+\delta+it,f)\big |^2 \\
       & \hskip 150pt
       = \mathcal{T}_1 + \mathcal{T}_2
             +  \mathcal{O}_{\varepsilon}\Big(q^{-\kappa/2+\varepsilon}+q^{-2\vartheta+\varepsilon}\Big),
    \end{split}
  \end{equation}
  where
  \[
    \begin{split}
      \mathcal{T}_1 & :=  \zeta_q(1+2\delta)
           \sum_{\ell_1=1}^\infty\sum_{\ell_2=1}^\infty
           \sum_{g=1}^\infty \frac{1}{g} \underset{h_2}{{\sum}^{\flat}} \frac{1}{h_2^2}
            \\
        & \hskip 70pt \cdot \frac{x_{g\ell_1h_2}(1/2+\delta+it)x_{g\ell_2h_2}(1/2+\delta-it)}{(\ell_1\ell_2)^{1+\delta}}
          \underset{h_1|\ell_1\ell_2}{{\sum}^{\flat}} \frac{1}{h_1} \sum_{d} \frac{\eta_{it}(d^2\ell_1\ell_2)}{d^{2+2\delta}}, \\
      \mathcal{T}_2 & := \zeta_q(1-2\delta) \frac{H_t(-\delta)}{H_t(\delta)} \left(\frac{q}{4\pi^2}\right)^{-2\delta}
           \sum_{\ell_1=1}^\infty\sum_{\ell_2=1}^\infty
           \sum_{g=1}^\infty \frac{1}{g} \underset{h_2}{{\sum}^{\flat}} \frac{1}{h_2^2}
           \\
      & \hskip 70pt \cdot \frac{x_{g\ell_1h_2}(1/2+\delta+it)x_{g\ell_2h_2}(1/2+\delta-it)}{(\ell_1\ell_2)^{1-\delta}}
           \underset{h_1|\ell_1\ell_2}{{\sum}^{\flat}} \frac{1}{h_1} \sum_{d}
              \frac{\eta_{it}(d^2\ell_1\ell_2)}{d^{2-2\delta}}.
    \end{split}
  \]
  To estimate the main term contributions, we will need the following lemma.

  \begin{lemma}\label{lemma:DirichletSeries}
    Let $\ell_1$ and $\ell_2$ be square-free. For $\Re(s\pm 2\nu)>1$ we have
    \[
      \begin{split}
        & \sum_{d} \frac{\eta_{\nu}(d^2\ell_1\ell_2)}{d^s} = \frac{\zeta(s)\zeta(s+2\nu)\zeta(s-2\nu)}{\zeta(2s)}
        \prod_{p\big|\frac{\ell_1\ell_2}{(\ell_1,\ell_2)^2}}
            \frac{\eta_{\nu}(p)}{1+p^{-s}}
        \prod_{p|(\ell_1,\ell_2)} \frac{\eta_\nu(p^2)-p^{-s}}{1+p^{-s}} .
      \end{split}
    \]
  \end{lemma}

   \begin{proof}[Proof of lemma \ref{lemma:DirichletSeries}]
   See Hough \cite[lemma 6.3]{hough2012zero}.
    \end{proof}

  By \eqref{eqn:x_l(s)}, we have
  \[
    \begin{split}
      \mathcal{T}_1 & =  \zeta_q(1+2\delta)
           \sum_{\ell_1=1}^\infty\sum_{\ell_2=1}^\infty
           \sum_{g=1}^\infty \frac{1}{g} \underset{h_2}{{\sum}^{\flat}} \frac{1}{h_2^2}
           \frac{\mu(g\ell_1h_2)\mu(g\ell_2h_2)}{(gh_2)^{2\delta} \ell_1^{1+2\delta+it} \ell_2^{1+2\delta-it}}
            \underset{h_1|\ell_1\ell_2}{{\sum}^{\flat}} \frac{1}{h_1} \\
        & \hskip 10pt \cdot \sum_{n_1}\frac{\mu^2(g\ell_1h_2n_1)F_{\Upsilon,M}(g\ell_1h_2n_1)}{n_1^{1+2\delta+2it}} \sum_{n_2}\frac{\mu^2(g\ell_2h_2n_2)F_{\Upsilon,M}(g\ell_2h_2n_2)}{n_2^{1+2\delta-2it}}
          \sum_{d} \frac{\eta_{it}(d^2\ell_1\ell_2)}{d^{2+2\delta}},  \\
      \mathcal{T}_2 & = \zeta_q(1-2\delta) \frac{H_t(-\delta)}{H_t(\delta)} \left(\frac{q}{4\pi^2}\right)^{-2\delta}
           \sum_{\ell_1=1}^\infty\sum_{\ell_2=1}^\infty
           \sum_{g=1}^\infty \frac{1}{g} \underset{h_2}{{\sum}^{\flat}} \frac{1}{h_2^2}
           \frac{\mu(g\ell_1h_2)\mu(g\ell_2h_2)}{(gh_2)^{2\delta} \ell_1^{1+it} \ell_2^{1-it}}
            \underset{h_1|\ell_1\ell_2}{{\sum}^{\flat}} \frac{1}{h_1}
           \\
      & \hskip 20pt \cdot
      \sum_{n_1}\frac{\mu^2(g\ell_1h_2n_1)F_{\Upsilon,M}(g\ell_1h_2n_1)}{n_1^{1+2\delta+2it}} \sum_{n_2}\frac{\mu^2(g\ell_2h_2n_2)F_{\Upsilon,M}(g\ell_2h_2n_2)}{n_2^{1+2\delta-2it}}
           \sum_{d} \frac{\eta_{it}(d^2\ell_1\ell_2)}{d^{2-2\delta}}.
    \end{split}
  \]
  We first consider $\mathcal{T}_1$. By lemma \ref{lemma:DirichletSeries}, we have
  \[
    \begin{split}
      \mathcal{T}_1 & =
            \zeta_q(1+2\delta) \frac{\zeta(2+2\delta)}{\zeta(4+4\delta)} |\zeta(2+2\delta+2it)|^2
           \underset{r=gh}{{\sum}^{\flat}}  \frac{1}{g^{1+2\delta} h^{2+2\delta}}
            \\
        & \hskip 30pt \cdot \underset{(\ell_1\ell_2n_1n_2,r)=1}{\underset{(\ell_1,n_1)=1,(\ell_2,n_2)=1} {\underset{\ell_1,\ell_2,n_1,n_2}{{\sum}^{\flat}}}} \frac{\mu(\ell_1)\mu(\ell_2)}{\ell_1^{1+2\delta+it} \ell_2^{1+2\delta-it} n_1^{1+2\delta+2it} n_2^{1+2\delta-2it}}
          F_{\Upsilon,M}(\ell_1 n_1 r) F_{\Upsilon,M}(\ell_2 n_2 r) \\
        & \hskip 80pt \cdot
        \prod_{p\big|\frac{\ell_1\ell_2}{(\ell_1,\ell_2)^2}}
            \left(1+\frac{1}{p}\right)
            \left(\frac{\eta_{\nu}(p)}{1+p^{-2-2\delta}}\right)
        \prod_{p|(\ell_1,\ell_2)} \left(1+\frac{1}{p}\right)
         \left(\frac{\eta_\nu(p^2)-p^{-2-2\delta}}{1+p^{-2-2\delta}} \right).
    \end{split}
  \]
  Here the factor $(1+1/p)$ in the products is coming from the $h_1$-sum.
It follows from Mellin inversion that
  \[
    \begin{split}
      \mathcal{T}_1 & =
            \zeta_q(1+2\delta) \frac{\zeta(2+2\delta)}{\zeta(4+4\delta)} |\zeta(2+2\delta+2it)|^2
            \\
        & \hskip 30pt \cdot
        \frac{1}{(2\pi i)^2} \int\limits_{\alpha-i\infty}^{\alpha+i\infty}\int\limits_{\alpha-i\infty}^{\alpha+i\infty}
       \mathcal{G}_1(w_1,w_2,\delta,t)\Psi_{\Upsilon,M}(w_1)\Psi_{\Upsilon,M}(w_2) dw_1 dw_2,
    \end{split}
  \]
  where
  \[
    \begin{split}
      & \mathcal{G}_1(w_1,w_2,\delta,t) := \underset{r=gh}{{\sum}^{\flat}}
      \frac{1}{r^{1+2\delta+w_1+w_2} h}  \underset{(\ell_1\ell_2n_1n_2,r)=1}{\underset{(\ell_1,n_1)=1,(\ell_2,n_2)=1} {\underset{\ell_1,\ell_2,n_1,n_2}{{\sum}^{\flat}}}} \frac{\mu(\ell_1)\mu(\ell_2)}{\ell_1^{1+w_1+2\delta+it} \ell_2^{1+w_2+2\delta-it} }
       \\
      & \hskip 10pt \cdot
        \frac{1}{ n_1^{1+w_1+2\delta+2it} n_2^{1+w_2+2\delta-2it}}
        \left(\prod_{p|(\ell_1,\ell_2)} \left(1+\frac{1}{p}\right)^{-1}
         \left(\frac{(\eta_\nu(p^2)-p^{-2-2\delta})(1+p^{-2-2\delta})}{\eta_\nu(p)^2} \right) \right)
         \\
        & \hskip 60pt \cdot
         \left( \prod_{p|\ell_1}     \left(1+\frac{1}{p}\right)
            \left(\frac{\eta_{\nu}(p)}{1+p^{-2-2\delta}}\right) \right)
        \left( \prod_{p|\ell_2}     \left(1+\frac{1}{p}\right)
            \left(\frac{\eta_{\nu}(p)}{1+p^{-2-2\delta}}\right)  \right).
    \end{split}
  \]
  To write $\mathcal{G}_1(w_1,w_2,\delta,t)$ as an Euler product, we can use the following lemma.
  \begin{lemma}\label{lemma:MultiplicativeFunction2}
    Let $f_1,f_2,g_1,g_2,h,t$ be some multiplicative functions. Denote
    \[
      \begin{split}
        & F(k) = \sum_{k=r\ell_1n_1\ell_2n_2} \mu^2(r\ell_1n_1)\mu^2(r\ell_2n_2) h(r) f_1(\ell_1)f_2(\ell_2)
        g_1(n_1)g_2(n_2)
        \prod_{p|(\ell_1,\ell_2)}t(p).
      \end{split}
    \]
    Then $F$ is a multiplicative function.
  \end{lemma}
  \begin{proof}[Proof of lemma \ref{lemma:MultiplicativeFunction2}]
    Note that $F$ is supported on cubic free numbers.
    Let $k=ab^2$, where $a,b$ are square-free numbers.
    We only need to consider the decompostion of $k=r\ell_1n_1\ell_2n_2$ with
    $b|\ell_1n_1$ and $b|\ell_2n_2$;
    otherwise the contribution will be zero because of $\mu^2$.
    Write $a=r\ell_1'n_1'\ell_2'n_2'$ and $b=b_1b_1'=b_2b_2'$,
    where $\ell_1=b_1\ell_1'$, $\ell_2=b_2\ell_2'$, $n_1=b_1'n_1'$,
    and $n_2=b_2'n_2'$. Hence we have
    \[
      \begin{split}
        & F(k)=F(ab^2)= \sum_{a=r\ell_1'n_1'\ell_2'n_2'}h(r)
         f_1(\ell_1')f_2(\ell_2')g_1(n_1')g_2(n_2')  \\
        &\hskip 120pt \cdot \sum_{b=b_1b_1'=b_2b_2'} f_1(b_1)g_1(b_1')f_2(b_2)g_2(b_2')
         \prod_{p|(b_1,b_2)}t(p).
      \end{split}
    \]
    Since $b$ is square-free, it's not hard to check that $F$ is multiplicative.
  \end{proof}

  Now by lemma \ref{lemma:MultiplicativeFunction2}, we have
  \[
    \begin{split}
      & \mathcal{G}_1(w_1,w_2,\delta,t) \\
      & = \prod_{p} \Bigg( 1
        + \frac{1+1/p}{p^{1+2\delta+w_1+w_2}}
        -\frac{\eta_{it}(p)}{p^{1+2\delta+it+w_1}}\frac{1+1/p}{1+p^{-2-2\delta}}
        -\frac{\eta_{it}(p)}{p^{1+2\delta-it+w_2}}\frac{1+1/p}{1+p^{-2-2\delta}}
        \\
      & \hskip 40pt
        +\frac{1}{p^{1+2\delta+2it+w_1}}
        +\frac{1}{p^{1+2\delta-2it+w_2}}
        + \frac{1}{p^{2+4\delta+w_1+w_2}}
        + \frac{\eta_{it}(p^2)-p^{-2-2\delta}}{p^{2+4\delta+w_1+w_2}}\frac{1+1/p}{1+p^{-2-2\delta}}
         \\
      & \hskip 116pt
        -\frac{\eta_{it}(p)}{p^{2+4\delta-it+w_1+w_2}}\frac{1+1/p}{1+p^{-2-2\delta}}
        -\frac{\eta_{it}(p)}{p^{2+4\delta+it+w_1+w_2}}\frac{1+1/p}{1+p^{-2-2\delta}} \Bigg) .
     \end{split}
   \]
   Hence we get
   \[
     \begin{split}
       & \mathcal{G}_1(w_1,w_2,\delta,t) \\
      & = \frac{\zeta(4+4\delta)}{\zeta(2+2\delta)}
            \prod_{p} \Bigg( 1+\frac{1}{p^{2+2\delta}}
        + \frac{1}{p^{1+2\delta+w_1+w_2}}
        + \frac{1}{p^{3+4\delta+w_1+w_2}}
        + \frac{1}{p^{2+2\delta+w_1+w_2}}\\
        & \hskip 40pt
        + \frac{1}{p^{4+4\delta+w_1+w_2}}
        -\frac{\eta_{it}(p)}{p^{1+2\delta+it+w_1}}
        -\frac{\eta_{it}(p)}{p^{2+2\delta+it+w_1}}
        -\frac{\eta_{it}(p)}{p^{1+2\delta-it+w_2}}
        -\frac{\eta_{it}(p)}{p^{2+2\delta-it+w_2}}
        \\
      & \hskip 40pt
        +\frac{1}{p^{1+2\delta+2it+w_1}}
        +\frac{1}{p^{3+4\delta+2it+w_1}}
        +\frac{1}{p^{1+2\delta-2it+w_2}}
        +\frac{1}{p^{3+4\delta-2it+w_2}}
        + \frac{1}{p^{2+4\delta+w_1+w_2}}
         \\
      & \hskip 40pt
        + \frac{1}{p^{4+6\delta+w_1+w_2}}
        + \frac{\eta_{it}(p^2)-p^{-2-2\delta}}{p^{2+4\delta+w_1+w_2}}
        + \frac{\eta_{it}(p^2)-p^{-2-2\delta}}{p^{3+4\delta+w_1+w_2}}
         \\
      & \hskip 40pt
        -\frac{\eta_{it}(p)}{p^{2+4\delta-it+w_1+w_2}}
        -\frac{\eta_{it}(p)}{p^{3+4\delta-it+w_1+w_2}}
        -\frac{\eta_{it}(p)}{p^{2+4\delta+it+w_1+w_2}}
        -\frac{\eta_{it}(p)}{p^{3+4\delta+it+w_1+w_2}}\Bigg) \\
      & \hskip 10pt =: \frac{\zeta(1+2\delta+w_1+w_2)}
      {\zeta(1+2\delta+w_1)\zeta(1+2\delta+w_2)} \mathcal{H}_1(w_1,w_2,\delta,t).
    \end{split}
  \]
  Here the Euler product defining $\mathcal{H}_1(w_1,w_2,\delta,t)$ converges absolutely for
  \[
    \min\big\{\Re(w_1),\Re(w_2),\Re(w_1+w_2)\big\} > -1/2.
  \]
  We find that
  \[
    \mathcal{H}_1(0,0,\delta,t) = \frac{\zeta(4+4\delta)}{\zeta(2+2\delta)|\zeta(2+2\delta+2it)|^2}.
  \]
  Hence, by exactly the same argument as in \S\ref{subsec:S1}, we can prove that
  \begin{equation}\label{eqn:T1=}
    \mathcal{T}_1 = 1 + \frac{M^{-2(1-\Upsilon)\delta}-M^{-2\delta}}{(2\delta\Upsilon\log M)^2}
    + \mathcal{O}\left( \frac{\log\log q}{\log q}  M^{-2(1-\Upsilon)\delta} \right).
  \end{equation}
  And similarly, we can handle $\mathcal{T}_2$, getting
  \begin{equation}\label{eqn:T2=}
    \begin{split}
     \mathcal{T}_2 & = q^{-2\delta}
            \bigg| \frac{M^{-2\delta+2it}-M^{(1-\Upsilon)(-2\delta+2it)}}
            {\Upsilon(-2\delta+2it)\log M} \bigg|^2
             \\
      & \hskip 80pt
      - q^{-2\delta} \frac{M^{-2\delta(1-\Upsilon)}-M^{-2\delta}}{(2\delta\Upsilon\log M)^2}
      + \mathcal{O}\left(  \frac{\log\log q}{\log q}  M^{-2(1-\Upsilon)\delta} \right).
    \end{split}
  \end{equation}
  By \eqref{eqn:T12}, \eqref{eqn:T1=}, and \eqref{eqn:T2=}, we prove \eqref{eqn:SecondMoment-w(x)} in the theorem.

  \medskip
  To prove \eqref{eqn:SecondMoment-n}, we will apply lemma \ref{lemma:Kowalski-Michel}.
  Note that condition \eqref{eqn:condition1} is a consequence of the trivial bound
  $|M(1/2+\delta+it,f)|<M^{1/2+\varepsilon}$,
  the convexity bound $$L(1/2+\delta+it,f)\ll q^{1/4+\varepsilon}(1+|t|)^{1/2+\varepsilon},$$
  and the fact $\omega_f\ll (\log q)q^{-1}$.
  The result in theorem \ref{thm:MollificationS} guarantees condition \eqref{eqn:condition2} under our choice of the parameters.
  The proof is completed upon applying lemma \ref{lemma:Kowalski-Michel}.
 \end{proof}

\subsection{Away from the critical line}\label{subsec:SM_L}

Now, we want to remove the harmonic weights that appear in \S\ref{sec:mollification_away}.
Our main result in this subsection is the following theorem.
\begin{theorem}\label{thm:second_moment_natural}
 With notations as in \S\ref{sec:msm}, assume $s=\frac12+\delta+it$ where
  $$\frac{1}{\log q} \; \leq \; \delta \; \leq \; \frac12+\frac{10 \log\log q}{(1-\Upsilon)\log q}.$$ Then there exists
  an absolute constant $B>0$, such that for any $0<a<2(1-\Upsilon)$, we have
  \[
    \frac{1}{\mathcal A(q)}\sum_{f\in H_2(q)} \big(|LM(s,f)|^2 - 1\big) \,
    \ll_{\Upsilon,\vartheta,a,B} \; (1+|t|)^B M^{-a\delta}.
  \]
\end{theorem}

\begin{proof}[Proof of theorem \ref{thm:second_moment_natural}]
As in \S\ref{sec:mollification_away}, the proof follows easily from lemmas
\ref{lemma:SecondMoment}, \ref{lemma:FirstMoment}, stated and proved below.
\end{proof}

\begin{lemma}\label{lemma:SecondMoment}
  With notations as above, assume $s=\frac12+\delta+it$ where
  $\delta\geq \frac{1}{(2\log q)}$. Then for any $0<a<2(1-\Upsilon)$ we have
  \[
    \frac{1}{\mathcal A(q)}\sum_{f\in H_2(q)} |LM(s,f)-1|^2\,
    \ll_{\Upsilon,\vartheta,a,B}\; (1+|t|)^B M^{-a\delta}.
  \]
\end{lemma}

\begin{proof}
  Together with theorem \ref{thm:MollificationS2} and the fact
  that $\omega_f\ll \log q$, the proof of the above lemma is
  similar to Kowalski--Michel \cite[Corollary 1]{kowalski2000explicit}.
\end{proof}

\begin{lemma}\label{lemma:FirstMoment}
  With notations as above, assume that
  $\frac{1}{\log q}\leq \delta \leq \frac12+\frac{10 \log\log q}{(1-\Upsilon)\log q}$.
  Then for any $0<a<2(1-\Upsilon)$ we have
  \[
    \frac{1}{\mathcal A(q)} \sum_{f\in H_2(q)} \big(LM(s,f)-1\big) \,
    \ll_{\Upsilon,\vartheta,B} \; (1+|t|)^B M^{-a\delta}.
  \]
\end{lemma}

\begin{proof}[Proof of lemma \ref{lemma:FirstMoment}]
  We will first handle the case $\Re(s)= \frac12+\delta_0$ where we set\linebreak
  $\delta_0 = \frac12+\frac{10 \log\log q}{(1-\Upsilon)\log q}$.
  Put $B(s,f)=LM(s,f)-1$. As in the proof of lemma \ref{lemma:FirstMomentL_h},
  for $\Re(s)>1$, we have
  \[
    LM(s,f) = \sum_{n=1}^{\infty} \frac{\lambda_f(n)}{n^{s}} c(n),
    \quad \textrm{where}\quad
    c(n) = \sum_{d|n} \mu(d) F_{\Upsilon,M}(d),
  \]
  and
  \[
    \begin{split}
      B(s,f) & = T_1(s,f) - T_2(s,f).
    \end{split}
  \]

  We first estimate the contribution of the $T_2(s,f)$ terms.
  By Cauchy's inequality and lemma \ref{lemma:SecondMoment},
  for some constant $B>0$ we have
  \begin{align*}
      &  \frac{ \mathcal{A}\Big( \big\{ |T_2(s,f)| \big\};q \Big)}{\mathcal{A}(q)}
     \ll  X^{-\delta_0+\delta_1} \frac{ \mathcal{A}\Big( \Big\{  \int\limits_{(-\delta_0+\delta_1)} |\Gamma(w)| |B(w+s,f)| |dw| \Big\};q \Big)}{\mathcal{A}(q)}
     \\
     & \hskip 75pt \ll   (1+|t|)^B X^{-\delta_0+\delta_1}
     \ll (1+|t|)^B q^{-(1-2\vartheta)\delta_0}
     \ll (1+|t|)^B M^{-2(1-\Upsilon)\delta_0}.
  \end{align*}

  It remains now to estimate the $T_1$ contribution.
  As before, we have
  \[
    \begin{split}
      T_1(s,f)  = \sum_{M^{1-\Upsilon}<n\leq X(\log q)^2} \frac{\lambda_f(n)c(n)}{n^s} e^{-n/X}
      + \mathcal{O}\Big(q^{-B}\Big).
    \end{split}
  \]
  In order to bound $\frac{1}{\mathcal{A}(q)} \mathcal{A}\big( \big\{ T_1(s,f) \big\};q \big)$,
  for $M^{1-\Upsilon}<n\leq X(\log q)^2$, we consider
  \[
    \frac{ \mathcal{A}\big( \big\{ \lambda_f(n) \big\};q \big)}{\mathcal{A}(q)}
    = \frac{1}{\zeta(2)}   \sum_{f\in H_2(q)}  \omega_f \cdot L(1,\operatorname{sym}^2 f)  \lambda_f(n)
    + \mathcal O\left(\frac{(\log q)^5}{q}\right).
  \]
  We can replace $L(1,\operatorname{sym}^2 f)$ by the following Dirichlet series
  (see e.g. Iwaniec--Michel \cite[Lemma 3.1]{iwaniec2001second})
  \[
    L(1,\operatorname{sym}^2 f)
    = \sum_{l<q^{1+\varepsilon}} \frac{\lambda_f(l^2)}{l} V_1\left(\frac{l}{q}\right)
    + \frac{L_\infty(0,\sym^2f)}{qL_\infty(1,\sym^2f)}
    \sum_{l<q^{1+\varepsilon}} \lambda_f(l^2) V_0\left(\frac{l}{q}\right) + \mathcal O(q^{-A}),
  \]
  where $L_\infty(w,\sym^2f) := \pi^{-3w/2}\Gamma(\frac{w+1}{2})
  \Gamma(\frac{w+k-1}{2})\Gamma(\frac{w+k}{2})$ and
  \[
    V_w(y) := \frac{1}{2\pi i} \int\limits_{(2)} y^{-u}
    \frac{L_\infty(w+u,\sym^2f)}{L_\infty(w,\sym^2f)} \zeta^{(q)}(2w+2u) \frac{du}{u}.
  \]
  Here $\zeta^{(q)}(w)$ stands for the partial zeta function
  with local factors at primes of $q$ removed. Thus
  \[
    \begin{split}
      \frac{ \mathcal{A}\big( \big\{ \lambda_f(n) \big\};q \big)}{\mathcal{A}(q)}
      & = \frac{1}{\zeta(2)}   \sum_{l<q^{1+\varepsilon}} \frac{V_1(l/q)}{l}
    \sum_{f\in H_2(q)}  \omega_f \cdot \lambda_f(l^2) \lambda_f(n)  \\
      & \hskip -45pt + \frac{L_\infty(0,\sym^2f)}{q\zeta(2)L_\infty(1,\sym^2f)}
       \sum_{l<q^{1+\varepsilon}} V_0(l/q)
    \sum_{f\in H_2(q)}  \omega_f \cdot \lambda_f(l^2) \lambda_f(n)
    + \mathcal O\left(\frac{(\log q)^5}{q}\right).
    \end{split}
  \]
  We shall only deal with the first term above, since
  the second one can be handled similarly.
  For $l < M^{1-\Upsilon}$, by applying lemma \ref{lemma:PTF}, we arrive at
  \[
    \begin{split}
       & \frac{1}{\zeta(2)}   \sum_{l<M^{1-\Upsilon}} \frac{V_1(l/q)}{l}
         \bigg( \delta_{n,l^2} - 2\pi \sum_{c\,\equiv\,0\,(\mod \hskip-3ptq)} \frac{S(l^2,n;c)}{c}
         J_{1}\left(\frac{4\pi l\sqrt{n}}{c}\right) \bigg) .
    \end{split}
  \]
  The contribution from the first term above to
  $\frac{1}{\mathcal{A}(q)} \mathcal{A}\big( \big\{ T_1(s,f) \big\};q \big)$ is
  \[
    \ll \; q^\varepsilon \sum_{M^{(1-\Upsilon)/2}<l<M^{1-\Upsilon}}
    \frac{1}{l^3} \; \ll \; M^{-(1-\Upsilon)+\varepsilon}.
  \]
  By the Weil bound for the Kloosterman sums, we know that
  the contribution from the second term above to
  $\frac{1}{\mathcal{A}(q)} \mathcal{A}\big( \big\{ T_1(s,f) \big\};q \big)$ is
  \[
    \begin{split}
      & \ll q^\varepsilon \sum_{l<M^{1-\Upsilon}}\frac{1}{l}
    \sum_{M^{1-\Upsilon}<n<X(\log q)^2}\frac{1}{n}
    \sum_{c\equiv0(q)} \frac{(c,n,l^2)^{1/2}\tau(c)}{c^{1/2}} \cdot \frac{l\sqrt{n}}{c}
      \\
      &  \ll M^{1-\Upsilon} q^{-1+\varepsilon} \ll M^{-(1-\Upsilon)}.
    \end{split}
  \]
  Hence we get that the contribution from $l<M^{1-\Upsilon}$ to
  $\frac{1}{\mathcal{A}(q)} \mathcal{A}\big( \big\{ T_1(s,f) \big\};q \big)$ is
  $$\ll M^{-(1-\Upsilon)+\varepsilon}.$$

  Now we are going to treat the case $M^{1-\Upsilon}\leq l < q^{1+\varepsilon}$.
  We will use the following two large sieve inequalities
  (see \cite[Theorem 1]{duke1994bounds})
  \[
    \sum_{f\in H_2(q)} \omega_f \Big|\sum_{n\leq N} \lambda_f(n) a_n \Big|^2
    \ll q^\varepsilon (N/q+1) \sum_{n\leq N} |a_n|^2,
  \]
  and (see \cite[Theorem 5.1]{iwaniec2001second})
  \[
    \sum_{f\in H_2(q)} \omega_f \bigg|\sum_{l} \lambda_f(l^2) g\left(\frac{l}{L}\right)\bigg|^2
    \ll q^\varepsilon (X^2/q+X).
  \]
  Here $g$ is a smooth function supported in $[1,2]$ which satisfies $g^{(j)}(x)\ll1$.
  By applying a smooth partition of unity to the $l$-sum, we know that the contribution
  of terms with $M^{1-\Upsilon}\leq l < q^{1+\varepsilon}$ is
  \[
    \begin{split}
      & \ll \sup_{M^{1-\Upsilon}\ll L \ll q^{1+\varepsilon}} \frac{1}{L}
      \sum_{f\in H_2(q)} \omega_f \sum_{l} \lambda_f(l^2) g\left(\frac{l}{L}\right)
      \sum_{M^{1-\Upsilon}<n\leq X(\log q)^2} \frac{\lambda_f(n)a_n}{n^s} \\
      & \ll \sup_{M^{1-\Upsilon}\ll L \ll q^{1+\varepsilon}} \frac{1}{L}
      \bigg(\sum_{f\in H_2(q)} \omega_f \bigg|\sum_{l} \lambda_f(l^2)
      g\left(\frac{l}{L}\right)\bigg|^2\bigg)^{1/2} \\
      & \hskip 120pt \cdot
      \bigg(\sum_{f\in H_2(q)} \omega_f \bigg| \sum_{M^{1-\Upsilon}<n\leq X(\log q)^2}
      \frac{\lambda_f(n)a_n}{n^s} \bigg|^2\bigg)^{1/2} \\
      & \ll M^{-(1-\Upsilon)+\varepsilon}.
    \end{split}
  \]
  Hence
  \[
    \frac{\mathcal{A}\big( \big \{ T_1(s,f) \big\};q \big)}{\mathcal{A}(q)}
     \ll M^{-(1-\Upsilon)+\varepsilon},
  \]
  and
  \[
   \frac{\mathcal{A}\big( \big\{(LM(1/2+\delta_0+it,f)-1) \big\};q \big)}{\mathcal{A}(q)}
    \ll_{B,a}  (1+|t|)^B M^{-(2(1-\Upsilon)-\varepsilon)\delta_0}.
  \]
 This completes the proof for $\Re(s)=1/2+\delta_0 = 1+\frac{10 \log\log q}{(1-\Upsilon)\log q}$.
  The result for $$\frac{1}{\log q}\; \leq \;\delta \leq 1/2+\frac{10 \log\log q}{(1-\Upsilon)\log q}$$
   follows from lemma \ref{lemma:SecondMoment} and the convexity argument.
\end{proof}


\section{Proof of theorem \ref{thm:RealZero}}\label{sec:thmrz}

The proofs 
depend on Selberg's lemma in \cite[Lemma 14]{selberg1946contributions2}
(see also \cite[Lemma 2.1]{conrey2002real}), and
are similar to those of \cite[theorems 1.8 and 1.5]{goldfeld2016super}.
One can see the similarity by replacing $K$ in \cite{goldfeld2016super} with $q^{1/2}$,
so we don't give the details.
In this section, we will prove that for at least $49\%$
of the modular L-functions of weight $2$ and level $q$ we have $L(\sigma,f)>0$
for $\sigma\in(1/2,1]$.

Let $M=q^{(1-5\vartheta)/2}$.
We apply Selberg's lemma with the choices
$$
  H=\frac{S}{\log q}, \quad
  W_0 = \frac{1}{2}-\frac{R}{\log q},\quad
  W_1=1+\frac{10 \log\log q}{(1-\Upsilon)\log q},\quad
  \phi(s) = LM(s, f),
$$
where $R$ and $S$ are fixed positive parameters which will be chosen later.
It follows that

\begin{equation}\label{eqn:sumtoI}
  4S \sum_{\substack{\beta\geq\frac{1}{2}-\frac{R}{\log q}\\0\leq\gamma\leq \frac{2S}{3\log q} \\ L(\beta+i\gamma,f)=0}}
  \cos\left(\frac{\pi \,\gamma\, \log q}{2S}   \right)
        \sinh\left(\frac{\pi \big(R+(\beta-1/2)\log q\big)}{2S}\right)
 \; \leq \; I_1(f)+I_2(f)+I_3(f),
\end{equation}
where
\begin{equation}\label{eqn:I(f)}
  \begin{split}
     I_1(f) \; & := \int\limits_{-S}^{S} \cos\left(\frac{\pi t}{2S}\right)
        \log\left|LM\left(\frac{1}{2}-\frac{R}{\log q}+i\frac{t}{\log q},f\right)\right|dt,  \\
     I_2(f) & \; := \int\limits_{-R}^{(W_1-1/2)\log q} \sinh\left(\frac{\pi(u+R)}{2S}\right)
        \log\left|LM\left(\frac{1}{2}+\frac{u}{\log q}+i\frac{S}{\log q},f\right)\right|^2 du, \\
     I_3(f) & \; := -\Re \int\limits_{-S}^{S} \cos\left(\pi\frac{(W_1-1/2)\log q -R+it}{2iS}\right)
        \log LM\left(W_1+i\frac{t}{\log q},f\right) dt.
  \end{split}
\end{equation}

Now, in the sum over zeros on the left hand side of (\ref{eqn:sumtoI}),
the weight $\cos\left(\frac{\pi\gamma}{2H}\right)$
can be replaced by $1$ (see \cite[eq. (7.3)]{goldfeld2016super}).
It follows that
\begin{equation}\label{eqn:sumtoI2}
  4S \sum_{\substack{\beta\geq\frac{1}{2}-\frac{R}{\log q}\\0\leq\gamma\leq \frac{2S}{3\log q} \\ L(\beta+i\gamma,f)=0}}
        \sinh\left(\frac{\pi \big(R+(\beta-1/2)\log q\big)}{2S}\right)
 \; \leq \; I_1(f)+I_2(f)+I_3(f).
\end{equation}

\begin{claim}\label{claim}
  We have
  \begin{equation*}
    I_1(f)+I_2(f)+I_3(f) \; \geq \; \begin{cases}
      4S\sinh\left(\frac{\pi R}{2S}\right), & \textrm{for all odd forms $f\in H_2(q),$}  \\
      & \\
      12S\sinh\left(\frac{\pi R}{2S}\right), &
      \begin{matrix}
        \textrm{for all odd forms $f\in H_2(q)$ and} \\
        \textrm{ if $L(s,f)$ has a zero $\rho=\beta+i\gamma$}\\
        \textrm{with $\beta\in (1/2,1]$ and $|\gamma|\leq \frac{2S}{3\log q}$,}
      \end{matrix} \\
      & \\
      8S\sinh\left(\frac{\pi R}{2S}\right), &
      \begin{matrix}
        \textrm{for all even forms $f\in H_2(q)$ and} \\
        \textrm{ if $L(s,f)$ has a zero $\rho=\beta+i\gamma$}\\
        \textrm{with $\beta\in (1/2,1]$ and $|\gamma|\leq \frac{2S}{3\log q}$.}
      \end{matrix}
    \end{cases}
  \end{equation*}
\end{claim}

\begin{proof}[Proof of Claim]
  This  is analogous to \cite[Claim 7.4]{goldfeld2016super}.
  Note that for odd forms $f \in H_2(q)$, we have $L(1/2, f) = 0$ for all these $f.$
\end{proof}

Let us now define
$$\mathcal N_r(q) \; := \underset{\text{for some} \;\beta\in (1/2,1], \; |\gamma|\leq \frac{2S}{3\log q} } {\underset{L(\beta+i\gamma,f)\, = \, 0\;  } {\sum_{f\in H_2(q)}}}   1.$$
By \eqref{eqn:sum-odd} and claim \ref{claim}, we have
\[
  \frac{1}{4}\mathcal A(q) + \mathcal N_{r}(q)
\;  \leq \;\frac{
        \cA\big(\big\{I_1(f)+I_2(f)+I_3(f)\big\}; q\big)}{8S\sinh\left(\frac{\pi R}{2S}\right)}.
\]
That is
\[
  \frac{\cN_{r}(q)}{\cA(q)}
 \; \leq \;
        \frac{\cA\big(\big\{I_1(f)+I_2(f)+I_3(f)\big\}; q\big)}
        {8S\sinh\left(\frac{\pi R}{2S}\right)\, \cA(q)}
        -\frac{1}{4}.
\]
It follows from theorems \ref{thm:MollificationS2} and \ref{thm:second_moment_natural}, with the same argument as in \cite[\S7]{goldfeld2016super},
 that
\begin{equation}\label{eqn:N0toV}
  \begin{split}
    \frac{\cN_{r}(q)}{\cA(q)}
    & \; \leq \; \frac{1}{8S\sinh\left(\frac{\pi R}{2S}\right)}
        \bigg( \int\limits_{0}^{S}  \cos\left(\frac{\pi t}{2S}\right)
        \log\Big( \sV(-R,t) \Big) \;dt \\
    & \hskip 30pt      +
        \int\limits_{0}^{\infty} \sinh\left(\frac{\pi u}{2S}\right)
        \log\Big( \sV(u-R,S) \Big) \;du \bigg)
     -\frac{1}{4} + \mathcal{O}_{c}\Big((\log q)^{-c}\Big),
  \end{split}
\end{equation}
for some constant $c>0$,
where
\begin{equation}\label{eqn:V(u,v)}
  \begin{split}
    \sV(u,v) & := 1 + e^{-2u}\left |\frac{e^{(-u+iv)(1-5\vartheta)}-e^{(1-\Upsilon)(-u+iv)(1-5\vartheta)}}{\Upsilon (-u+iv)(1-5\vartheta)}\right |^2 \\
    & \hskip 90pt + (1-e^{-2u}) \frac{e^{-u(1-\Upsilon)(1-5\vartheta)}-e^{-u(1-5\vartheta)}}{(u\Upsilon(1-5\vartheta))^2} .
  \end{split}
\end{equation}
Here we set $\vartheta=10^{-10}$ and $S=\frac{\pi}{2(1-\Upsilon)(1-20\vartheta)}$.
Then, by taking $\Upsilon=0.48$ and $R=7$, and a computer calculation of the integrals on the left hand side of (\ref{eqn:N0toV}), we get
\begin{equation}\label{eqn:Nr}
  \frac{\cN_{r}(q)}{\cA(q)} \;\leq\; 0.5041,
\end{equation}
when $q$ is sufficiently large.
Hence the theorem follows easily.

\section{Proof of theorem \ref{thm:SuperPositivity} } \label{sec:thmsp}

In this section we will take  $\vartheta=10^{-10}$, $\Upsilon=0.64$,
$R=4.6$ and $S=\frac{\pi}{2(1-\Upsilon)(1-20\vartheta)}$.
Let $J:=[C\log\log q]$, where $C$ is a large constant and $[x]$ means the largest integer less than $x$.  Set
$d := 2S/3$
and define the regions:
$$
  \mathcal{R}_j :=
  \begin{cases}
    \Big\{\beta+i\gamma \; \Big| \; \beta>\frac{1}{2}, \;\;  |\gamma|\leq \frac{d}{\log q}\Big\}, & \qquad  \mathrm{if} \; j=0,\\
  &\\
    \Big\{\beta+i\gamma \; \Big| \; \beta\geq \frac{1}{2}+\frac{jd}{\log q}, \;\;  |\gamma|\leq \frac{(j+1)d}{\log q}\Big\}, & \qquad  \mathrm{if} \; 1\leq j \leq J-1,\\
  &\\
   \Big\{\beta+i\gamma \; \Big| \; \beta\geq \frac{1}{2}+\frac{Jd}{\log q}, \;\;  |\gamma|\leq 1\Big\}, & \qquad \text{if} \; j = J,
   \end{cases}
$$
and the zero counting sum
$$\cN_j(q) \; :=
 \underset{ \text{one zero in}  \,\mathcal R_{j}} {\underset{L(s,f)\; \text{has at least} } {\sum_{f\in H_2(q)}}} \hskip-13pt 1, \qquad\;\; (\mathrm{where}\; 0\leq j \leq J).$$

For $j=0$, with our new choice of $R$ and $\Upsilon$, by the method in \S\ref{sec:thmrz}, we have
\begin{equation}\label{eqn:N0}
  \frac{\mathcal N_0(q)}{\mathcal A(q)} \leq 0.60934.
\end{equation}

For $1\leq j\leq J-1$, let $\cB_j$ be the rectangular box with vertices
$W_{0,j}\pm H_j$ and $W_1\pm H_j$,
where
\begin{equation}\label{eqn:W0j&Hj}
  W_{0,j} := \frac{1}{2}+\frac{jd/2}{\log q},\quad
  H_j := \frac{3(j+1)d/2}{\log q}.
\end{equation}
By Selberg's lemma, and the argument in \S\ref{sec:thmrz},
we have
\begin{equation}\label{eqn:Nj<}
  \begin{split}
    &  \frac{\cN_{j}(q)}{\cA(q)}   \; \leq \; \frac{1}{6(j+1)d\sinh\left(\frac{\pi j}{6(j+1)}\right)}
        \left[\int\limits_{0}^{\frac{3}{2}(j+1)d} \cos\left(\frac{\pi t}{3(j+1)d}\right)
        \log \Big( \sV(jd/2,t) \Big) dt \right. \\
    & \hskip 20pt
        + \left. \int\limits_{0}^{\infty} \sinh\left(\frac{\pi u}{3(j+1)d}\right)
        \log \Big( \sV(u+jd/2,3(j+1)d/2) \Big) du \phantom{\int\limits_{0}^{\frac{3}{2}(j+1)d}} \hskip -27pt\right]
        + \; \mathcal{O}_{c}\Big((\log q)^{-c}\Big),
  \end{split}
\end{equation}
for some constant $c>0$.
By a computer calculation of the integrals on the right hand side of (\ref{eqn:Nj<}), we obtain the following bounds:
\begin{equation}\label{eqn:Nj<-s}
  \begin{split}
   \frac{\cN_{1}(q)}{\cA(q)} \;  \leq \; 0.21032, \qquad
   \frac{\cN_{2}(q)}{\cA(q)} & \; \leq \; 0.03758, \qquad
   \frac{\cN_{3}(q)}{\cA(q)} \;  \leq \; 0.00995, \\
   \sum_{j=4}^{20} \;\frac{\cN_{j}(q)}{\cA(q)}  & \; \leq  \; 0.00528,
  \end{split}
\end{equation}
provided $q$ is sufficiently large.

To obtain similar bounds for $21\leq j \leq  J-1$, we will, instead,
use the following trivial bound for $u\geq 20$
\[ 1\leq \sV(u,v)  \leq 1 + e^{-0.35u}. \]
In this case
\begin{equation}\label{eqn:Nj<-l}
  \sum_{21\leq j\leq J-1} \frac{\cN_{j}(q)}{\cA(q)}
  \leq  0.001
  \end{equation}
for suitable choice of $C$.
Note that by \cite[Theorem 4]{kowalski1999analytic} we also have
\begin{equation}\label{eqn:N_J}
  \frac{\cN_{J}(q)}{\cA(q)} \ll (\log q)^{-c},
\end{equation}
when $C$ is large enough.
Thus by \eqref{eqn:N0}, \eqref{eqn:Nj<-s}, \eqref{eqn:Nj<-l}, and \eqref{eqn:N_J},
we get
\[
  \frac{\cN(q)}{\cA(q)}
\;  \leq \; \sum_{j=0}^{J} \frac{\cN_{j}(q)}{\cA(q)}
\;  \leq \; 0.88,
\]
for sufficiently large $q$.
It immediately follows from the above that
\begin{equation}\label{bound1}
  \cM(q)=\cA(q)-\cN(q)
\; \geq \; 0.12 \cdot\cA(q).
\end{equation}
This completes the proof of our theorem.



\begin{thebibliography}{99}


\vskip 3pt
\bibitem[CS02]{conrey2002real}
J.~B. Conrey and K.~Soundararajan.
\newblock Real zeros of quadratic {D}irichlet {L}-functions.
\newblock {\em Invent. Math.}, 150(1):1--44, 2002.

\vskip 3pt
\bibitem[CS]{conrey3000real}
J.~B. Conrey and K.~Soundararajan.
\newblock Real zeros of {L}-functions of modular forms.
\newblock {\em Private communication}.


\vskip 3pt
\bibitem[DFI94]{duke1994bounds}
W.~Duke, J.~B. Friedlander, and H.~Iwaniec.
\newblock Bounds for automorphic {$L$}-functions. {II}.
\newblock {\em Invent. Math.}, 115(2):219--239, 1994.


\vskip 3pt
\bibitem[GH16]{goldfeld2016super}
D.~Goldfeld and B.~Huang.
\newblock Super-positivity of a family of {L}-functions.
\newblock {\em arXiv preprint arXiv:1612.09359}, 2016.


\vskip 3pt
\bibitem[Hou12]{hough2012zero}
B.~Hough.
\newblock Zero-density estimate for modular form {L}-functions in weight
  aspect.
\newblock {\em Acta Arith.}, 154(2):187--216, 2012.

\vskip 3pt
\bibitem[Iwa97]{iwaniec1997topics}
H.~Iwaniec.
\newblock {\em Topics in classical automorphic forms}, volume~17.
\newblock American Mathematical Soc., 1997.

\vskip 3pt
\bibitem[IK04]{iwaniec2004analytic}
H.~Iwaniec and E.~Kowalski.
\newblock {\em Analytic number theory}, volume~53.
\newblock American Mathematical Society Providence, 2004.

\vskip 3pt
\bibitem[IM01]{iwaniec2001second}
H.~Iwaniec and P.~Michel.
\newblock The second moment of the symmetric square {$L$}-functions.
\newblock {\em Ann. Acad. Sci. Fenn. Math.}, 26(2):465--482, 2001.


\vskip 3pt
\bibitem[KM99]{kowalski1999analytic}
E.~Kowalski and P.~Michel.
\newblock The analytic rank of {$J_0(q)$} and zeros of automorphic
  {$L$}-functions.
\newblock {\em Duke Math. J.}, 100(3):503--542, 1999.

\vskip 3pt
\bibitem[KM00a]{kowalski2000explicit}
E.~Kowalski and P.~Michel.
\newblock Explicit upper bound for the (analytic) rank of {$J_0(q)$}.
\newblock {\em Israel J. Math.}, 120(part A):179--204, 2000.

\vskip 3pt
\bibitem[KM00b]{kowalski2000lower}
E.~Kowalski and P.~Michel.
\newblock A lower bound for the rank of {$J_0(q)$}.
\newblock {\em Acta Arith.}, 94(4):303--343, 2000.


\vskip 3pt
\bibitem[Sel46]{selberg1946contributions2}
A.~Selberg.
\newblock Contributions to the theory of {D}irichlet's {$L$}-functions.
\newblock {\em Skr. Norske Vid. Akad. Oslo. I.}, 1946(3):62, 1946.


\vskip 3pt
\bibitem[SZ80]{StarkZagier1980}
H. M.~Stark and D.~Zagier.
\newblock A property of L-functions on the real line.
\newblock  J. Number Theory, 12(1):49Ð52, 1980.


\vskip 3pt
\bibitem[YZ15]{yun2015shtukas}
Z.~Yun and W.~Zhang.
\newblock Shtukas and the Taylor expansion of $L$-functions.
\newblock {\em Ann. Math.}, Vol. 187, 767--911, 2018.

\end{thebibliography}
\end{document}